\documentclass{amsart}

\ifx\MyBeamer\undefined
\usepackage[breaklinks]{hyperref}
\usepackage{fullpage}
\else
\if\MyBeamer t
\usepackage[breaklinks]{hyperref}
\usepackage{beamerarticle}
\usepackage{fullpage}
\fi\fi

\usepackage{amsmath}
\usepackage{amssymb}
\usepackage{mathrsfs}
\usepackage{amsthm}
\usepackage{aliascnt}
\usepackage[utf8]{inputenc}
\usepackage[T1]{fontenc}
\usepackage[nofancy]{svninfo}
\ifx\FrenchText\undefined
\usepackage[british]{babel}
\else
\usepackage[british,french]{babel}
\AtBeginDocument{%
\catcode`\:=12%
\catcode`\!=12%
}
\fi

\makeatletter

\ifx\MyBeamer\undefined

\def\equationautorefname~#1\null{(#1)}
\def\itemautorefname~#1\null{#1}
\fi

\ifx\MyBeamer\undefined

\ifx\thmnum\undefined
\newtheorem{thmnum}{}[section]
\fi

\newcommand{\mynewthm}[3][]{%
  \newaliascnt{#2}{thmnum}%
  \newtheorem{#2}[#2]{#3}%
  \aliascntresetthe{#2}%
  \newtheorem*{#2*}{#3}%
  \expandafter\newcommand\csname #2autorefname\endcsname{#3}%
  \expandafter\renewcommand\csname the#2\endcsname{\thethmnum}%
}

\newtheorem*{clm}{Claim}
\newenvironment{clmprf}{%
  \begin{proof}[Proof of claim]%
  }{\end{proof}}

\else

\let\xxx=\frametitle
\def\frametitle#1{%
  \xxx{%
    \setbeamercolor*{math text}{use={titlelike,my math text},fg=titlelike.fg!80!my math text.fg}%
    #1}%
  \setbeamercolor{math text}{use=my math text,fg=my math text.fg}%
}

\newcommand{\beamerenv}[3]{%
\newenvironment<>{#1}%
{%
  \setbeamercolor{temp}{structure}%
  \setbeamercolor{structure}{fg=#2}%
  \setbeamercolor{block body}{use=structure,bg=structure.fg!5!white}%
  \begin{#3}%
}%
{\end{#3}\setbeamercolor{structure}{temp}}}

\newcommand{\mynewthm}[3][green!50!black]{%
  \newtheorem*{#2x}{#3}%
  \beamerenv{#2}{#1}{#2x}%
}

\beamerenv{other}{violet}{block}

\newtheorem*{clm}{Claim}

\fi

\ifx\FrenchText\undefined
\newcommand{\myiffrench}[2]{#2}
\else
\newcommand{\myiffrench}[2]{\iflanguage{french}{#1}{#2}}
\fi

\theoremstyle{plain}
\mynewthm[red]{thm}{\myiffrench{Théorème}{Theorem}}
\mynewthm{prp}{Proposition}
\mynewthm[purple]{lem}{\myiffrench{Lemme}{Lemma}}
\mynewthm[orange]{fct}{\myiffrench{Fait}{Fact}}
\mynewthm[purple!50!white]{cor}{\myiffrench{Corollaire}{Corollary}}

\theoremstyle{definition}
\mynewthm[green!80!black]{dfn}{\myiffrench{Définition}{Definition}}
\mynewthm{hyp}{\myiffrench{Hypothèse}{Hypothesis}}
\mynewthm{conv}{Convention}
\mynewthm{conj}{Conjecture}
\mynewthm{ntn}{Notation}
\mynewthm{cst}{Construction}

\theoremstyle{remark}
\mynewthm[yellow!90!black]{rmk}{\myiffrench{Remarque}{Remark}}
\mynewthm{qst}{Question}
\mynewthm{exc}{\myiffrench{Exercice}{Exercise}}
\mynewthm{exm}{\myiffrench{Exemple}{Example}}

\ifx\MyBeamer\undefined

\newcommand{\myenumlabel}[1]{\textnormal{(\roman{#1})}}

\fi

\newcounter{cycprfcnt}
\newcommand{\cycprfpreamble}%
{%
  \setcounter{cycprfcnt}{1}
  \setlength{\itemindent}{0.5\leftmargin}%
  \setlength{\leftmargin}{0pt}%
  \newcommand{\cpcurr}{\myenumlabel{cycprfcnt}}%
  \newcommand{\cpnext}{\addtocounter{cycprfcnt}{1}\cpcurr}%
  \newcommand{\cpnum}[1]{\setcounter{cycprfcnt}{##1}\cpcurr}%
  \newcommand{\cpfirst}{\cpnum{1}}%
  \newcommand{\impnext}{\cpcurr{} $\Longrightarrow$ \cpnext.}%
  \newcommand{\impfirst}{\cpcurr{} $\Longrightarrow$ \cpfirst.}%
}%

\newenvironment{cycprf}%
{\begin{list}{\impnext}%
  {\cycprfpreamble}}%
{\qedhere\end{list}}%

\newenvironment{cycprf*}%
{\begin{list}{\impnext}%
  {\cycprfpreamble}}%
{\end{list}}%

\def\indsym#1#2{%
  \setbox0=\hbox{$\m@th#1x$}%
  \kern\wd0%
  \hbox to 0pt{\hss$\m@th#1\mid$\hbox to 0pt{$\m@th#1^{#2}$\hss}\hss}%
  \lower.9\ht0\hbox to 0pt{\hss$\m@th#1\smile$\hss}%
  \kern\wd0}
\newcommand{\ind}[1][]{\mathop{\mathpalette\indsym{#1}}}

\def\nindsym#1#2{%
  \setbox0=\hbox{$\m@th#1x$}%
  \kern\wd0%
  \hbox to 0pt{\hss$\m@th#1\not$\kern1.4\wd0\hss}
  \hbox to 0pt{\hss$\m@th#1\mid$\hbox to 0pt{$\m@th#1^{#2}$\hss}\hss}%
  \lower.9\ht0\hbox to 0pt{\hss$\m@th#1\smile$\hss}%
  \kern\wd0}

\def\dotminussym#1#2{%
  \setbox0=\hbox{$\m@th#1-$}%
  \kern.5\wd0%
  \hbox to 0pt{\hss\hbox{$\m@th#1-$}\hss}%
  \raise.6\ht0\hbox to 0pt{\hss$\m@th#1.$\hss}%
  \kern.5\wd0}

\renewcommand{\emptyset}{\varnothing}
\renewcommand{\setminus}{\smallsetminus}

\DeclareMathOperator{\tp}{tp}

\DeclareMathOperator{\acl}{acl}

\DeclareMathOperator{\cl}{cl}

\makeatother

\usepackage{verbatim}
\usepackage{graphicx}

\begin{document}

\global\long\def\acl{\operatorname{acl}}
\global\long\def\Avg{\operatorname{Avg}}
\global\long\def\inp{\operatorname{inp}}
\global\long\def\EM{\operatorname{EM}}
\global\long\def\ist{\operatorname{ist}}
\global\long\def\M{\operatorname{\mathbb{M}}}
\global\long\def\NTP{\operatorname{NTP}}
\global\long\def\NIP{\operatorname{NIP}}
\global\long\def\TP{\operatorname{TP}}
\global\long\def\tp{\operatorname{tp}}
\global\long\def\transp{\operatorname{T}}
\global\long\def\lstp{\operatorname{L}}
\global\long\def\NSOP{\operatorname{NSOP}}
\global\long\def\bdn{\operatorname{bdn}}
\global\long\def\cl{\operatorname{cl}}
\global\long\def\dprk{\operatorname{dp-rk}}
\global\long\def\fund{\operatorname{fund}}
\global\long\def\div{\operatorname{div}}
\global\long\def\card{\operatorname{\mbox{Card}^{*}}}

%%%%%%%%%%%%%%%%%%%%%%%%%%%%%%%%%%%%%%%%%%%%%%%%%%%%%%%%%%%%%%%%%%%%%%%%

\title{An independence theorem for NTP$_2$ theories}

\author{Itaï \textsc{Ben Yaacov}}

\address{Itaï \textsc{Ben Yaacov} \\
  Université Claude Bernard -- Lyon 1 \\
  Institut Camille Jordan, CNRS UMR 5208 \\
  43 boulevard du 11 novembre 1918 \\
  69622 Villeurbanne Cedex \\
  France}

\urladdr{\url{http://math.univ-lyon1.fr/~begnac/}}

\author{Artem \textsc{Chernikov}}

\address{Artem \textsc{Chernikov} \\
  Einstein Institute of Mathematics \\
  Edmond J. Safra Campus, Givat Ram \\
  The Hebrew University of Jerusalem \\
  Jerusalem, 91904 \\
  Israel
}

\urladdr{\url{http://chernikov.me}}

\thanks{First author supported by the Institut Universitaire de France}
\thanks{Second author supported by the Marie Curie Initial Training
  Network in Mathematical Logic - MALOA - From MAthematical LOgic to
  Applications, PITN-GA-2009-238381}

\begin{abstract}
  We establish several results regarding dividing and forking in $\NTP_2$ theories.

  We show that dividing is the same as array-dividing.
  Combining it with existence of strictly invariant sequences we deduce that forking satisfies the chain condition over extension bases (namely, the forking ideal is $S1$, in Hrushovski's terminology).
  Using it we prove an independence theorem over extension bases (which, in the case of simple theories, specializes to the ordinary independence theorem).
  As an application we show that Lascar strong type and compact strong type coincide over extension bases in an $\NTP_2$ theory.

  We also define the dividing order of a theory -- a generalization of Poizat's fundamental order from stable theories -- and give some equivalent characterizations under the assumption of $\NTP_2$.
  The last section is devoted to a refinement of the class of strong theories and its place in the classification hierarchy.
\end{abstract}

\maketitle

\section*{Introduction}

The class of $\NTP_2$ theories, namely theories without the tree
property of the second kind, was introduced by Shelah \cite{Shelah:SimpleUnstableTheories}
and is a natural generalization of both simple and NIP theories containing
new important examples (e.g. any ultra-product of $p$-adics is $\NTP_2$,
see \cite{Chernikov:NTP2}).

The realization that it is possible to develop a good theory of forking
in the $\NTP_2$ context came from the paper \cite{Chernikov-Kaplan:ForkingDividingNTP2}, where
it was demonstrated that the basic theory can be carried out as long
as one is working over an extension base (a set is called an extension
base if every complete type over it has a global non-forking extension,
e.g. any model or any set in a simple, o-minimal or C-minimal theory
is an extension base).

Here we establish further important properties of forking, thus demonstrating
that a large part of simplicity theory can be seen as a special case
of the theory forking in $\NTP_2$ theories.

~

In \autoref{sec: Array-dividing} we consider the notion of \emph{array
  dividing}, which is a multi-dimensional generalization of dividing.
We show that in an $\NTP_2$ theory, dividing coincides with array
dividing over an arbitrary set (thus generalizing a corresponding
result of Kim for the class of simple theories).

\autoref{sec:Chain-condition} is devoted to a property of forking
called the  \emph{chain condition}. We say that forking in $T$ satisfies
the chain condition over a set $A$ if for any $A$-indiscernible
sequence $\left(a_{i}\right)_{i\in\omega}$ and any formula $\varphi\left(x,y\right)$,
if $\varphi\left(x,a_{0}\right)$ does not fork over $A$, then $\varphi\left(x,a_{0}\right)\land\varphi\left(x,a_{1}\right)$
does not fork over $A$. This property is equivalent to requiring
that there are no anti-chains of unbounded size in the partial order
of formulas non-forking over $A$ ordered by implication (hence the
name, see \autoref{sec:Chain-condition} for more equivalences
and the history of the notion). The following question had been raised
by Adler and by Hrushovski:
\begin{qst}
  What are the implications between $\NTP_2$ and the chain condition?
\end{qst}
We resolve it by showing that:
\begin{enumerate}
\item Forking in $\NTP_2$ theories satisfies the chain condition over
  extension bases (\autoref{thm: NTP2 implies Chain Condition},
  our proof combines the equality of dividing and array-dividing with
  the existence of universal Morley sequences from \cite{Chernikov-Kaplan:ForkingDividingNTP2}).
\item There is a theory with $\TP_{2}$ in which forking satisfies the chain
  condition (\autoref{sec: Chain condition does not imply NTP2}).
\end{enumerate}
In his work on approximate subgroups, Hrushovski \cite{Hrushovski:ApproximateSubgroups}
reformulated the independence theorem for simple theories with respect
to an arbitrary invariant $S1$-ideal. In \autoref{sec: Weak-independence-theorem}
we observe that the chain condition means that the forking ideal is
$S1$. Using it we prove a independence theorem for forking over
an arbitrary extension base in an $\NTP_2$ theory (\autoref{thm: Weak Independence Theorem}),
which is a natural generalization of the independence theorem of Kim
and Pillay for simple theories. As an application we show that Lascar
type coincides with compact strong type over an extension base in
an $\NTP_2$ theory.

In \autoref{sec: Dividing-order} we discuss a possible generalization
of the fundamental order of Poizat which we call the \emph{dividing
  order}. We prove some equivalent characterizations and connections
to the existence of universal Morley sequences in the case of $\NTP_2$
theories, and make some conjectures.

In the final section we define burden$^{2}$ and strong$^{2}$ theories
(which coincide with strongly$^{2}$ dependent theories under the
assumption of NIP, just as Adler's strong theories specialize to strongly
dependent theories). We establish some basic properties of burden$^{2}$
and prove that $\NTP_2$ is characterized by the boundedness of
burden$^{2}$.

\subsection*{Preliminaries}

We assume some familiarity with the basics of forking and dividing
(e.g. \cite[Section 2]{Chernikov-Kaplan:ForkingDividingNTP2}), simple theories (e.g. \cite{Wagner:SimpleTheories})
and NIP theories (e.g. \cite{Adler:IntroductionToDependent}).

As usual, $T$ is a complete first-order theory, $\M\vDash T$ is
a monster model. We write $a\ind_{C}b$ when $\tp(a/bC)$ does not
fork over $C$ and $a\ind_{C}^{d}b$ when $\tp(a/bC)$ does not divide
over $C$. In general these relations are not symmetric. We say that
a global type $p\left(x\right)\in S\left(\M\right)$ is \emph{invariant}
(\emph{Lascar-invariant}) over $A$ if whenever $\varphi\left(x,a\right)\in p$
and $b\equiv_{A}a$ (resp. $b\equiv_{A}^{\lstp}a$, see \autoref{def: Lascar strong type}), then $\varphi\left(x,b\right)\in p$.

We use the plus sign to denote concatenation of sequences, as in $I+J$,
or $a_{0}+I+b_{1}$ and so on.
\begin{dfn}
  \label{def: NTP2} Recall that a formula $\varphi\left(x,y\right)$
  is $\TP_{2}$ if there are $\left(a_{ij}\right)_{i,j\in\omega}$ and
  $k\in\omega$ such that:
  \begin{itemize}
  \item $\left\{ \varphi\left(x,a_{ij}\right)\right\} _{j\in\omega}$ is $k$-inconsistent
    for each $i\in\omega$,
  \item $\left\{ \varphi\left(x,a_{if\left(i\right)}\right)\right\} _{i\in\omega}$
    is consistent for each $f:\,\omega\to\omega$.
  \end{itemize}
\end{dfn}
A formula is $\NTP_2$ if it is not $\TP_{2}$, and a theory $T$
is $\NTP_2$ if it implies that every formula is $\NTP_2$.

\section{\label{sec: Array-dividing} Array dividing}

For the clarity of exposition (and since this is all that we will
need) we only deal in this section with $2$-dimensional arrays. All
our results generalize to $n$-dimensional arrays by an easy induction
(or even to $\lambda$-dimensional arrays for an arbitrary ordinal
$\lambda$, by compactness; see \cite[Section 1]{BenYaacov:SimplicityInCats}).
\begin{dfn}

  \begin{enumerate}
  \item We say that $\left(a_{ij}\right)_{i,j\in\kappa}$ is an \emph{indiscernible
      array} over $A$ if both $\left(\left(a_{ij}\right)_{j\in\kappa}\right)_{i\in\kappa}$
    and $\left(\left(a_{ij}\right)_{i\in\kappa}\right)_{j\in\kappa}$
    are indiscernible sequences. Equivalently, all $n\times n$ sub-arrays
    have the same type over $A$, for all $n<\omega$. Equivalently, $\tp(a_{i_{0}j_{0}}a_{i_{0}j_{1}}...a_{i_{n}j_{n}}/A)$
    depends just on the quantifier-free types of $\left\{ i_{0},...,i_{n}\right\} $
    and $\left\{ j_{0},...,j_{n}\right\} $ in the language of order and equality. Notice that, in particular,
    $\left(a_{if(i)}\right)_{i\in\kappa}$ is an $A$-indiscernible sequence
    of the same type for any strictly increasing function $f:\,\kappa\to\kappa$.
  \item We say that an array $\left(a_{ij}\right)_{i,j\in\kappa}$ is \emph{strongly
      indiscernible} over $A$ if it is an indiscernible array over $A$,
    and in addition its rows are mutually indiscernible over $A$, i.e.
    $\left(a_{ij}\right)_{j\in\kappa}$ is indiscernible over $\left(a_{i'j}\right)_{i'\in\kappa\setminus\left\{ i\right\} ,j\in\kappa}$
    for each $i\in\kappa$.
  \end{enumerate}
\end{dfn}

\begin{dfn}
  We say that $\varphi(x,a)$ \emph{array-divides} over $A$ if there
  is an $A$-indiscernible array $\left(a_{ij}\right)_{i,j\in\omega}$
  such that $a_{00}=a$ and $\left\{ \varphi(x,a_{ij})\right\} _{i,j\in\omega}$
  is inconsistent.
\end{dfn}

\begin{dfn}
  \begin{enumerate}
  \item Given an array $\mathbf{A}=\left(a_{i,j}\right)_{i,j\in\omega}$ and $k\in\omega$, we define:
    \begin{enumerate}
    \item $\mathbf{A}^{k}=\left(a_{i,j}'\right)_{i,j\in\omega}$ with $a_{i,j}' = a_{ik,j}, a_{ik+1,j}, \ldots, a_{ik+k-1,j}$.
    \item $\mathbf{A}^{\transp}=\left(a_{j,i}\right)_{i,j\in\omega}$, namely
      the transposed array.
    \end{enumerate}
  \item Given a formula $\varphi\left(x,y\right)$, we let $\varphi^{k}\left(x,y_{0}\ldots y_{k-1}\right)=\bigwedge_{i<k}\varphi\left(x,y_{i}\right)$.
  \item Notice that with this notation $\left(\mathbf{A}^{k}\right)^{l}=\mathbf{A}^{kl}$
    and $\left(\varphi^{k}\right)^{l}=\varphi^{kl}$.
  \end{enumerate}
\end{dfn}
\begin{lem}
  \label{lem: power arrays}
  \begin{enumerate}
  \item If $\mathbf{A}$ is a $B$-indiscernible array, then $\mathbf{A}^{k}$
    (for any $k\in\omega$) and $\mathbf{A}^{\transp}$ are $B$-indiscernible
    arrays.
  \item If $\mathbf{A}$ is a strongly indiscernible array over $B$, then $\mathbf{A}^{k}$
    is a strongly indiscernible array over $B$ (for any $k\in\omega$).
  \end{enumerate}
\end{lem}

\begin{lem}
  \label{lem: very indiscernible array is consistent} Assume that $T$
  is $\NTP_2$ and let $\left(a_{ij}\right)_{i,j\in\omega}$ be a
  strongly indiscernible array. Assume that the first column $\left\{ \varphi\left(x,a_{i0}\right)\right\} _{i\in\omega}$
  is consistent. Then the whole array $\left\{ \varphi\left(x,a_{ij}\right)\right\} _{i,j\in\omega}$
  is consistent.
\end{lem}
\begin{proof}
  Let $\varphi\left(x,y\right)$ and a strongly indiscernible array $\mathbf{A}=\left(a_{ij}\right)_{i,j\in\omega}$
  be given. By compactness, it is enough to prove that $\left\{ \varphi\left(x,a_{ij}\right)\right\} _{i<k,j\in\omega}$
  is consistent for every $k\in\omega$. So fix some $k$, and let $\mathbf{A}^{k}=\left(b_{ij}\right)_{i,j\in\omega}$
  --- it is still a strongly indiscernible array by \autoref{lem: power arrays}.
  Besides $\left\{ \varphi^{k}\left(x,b_{i0}\right)\right\} _{i\in\omega}$
  is consistent. But then $\left\{ \varphi^{k}\left(x,b_{ij}\right)\right\} _{j\in\omega}$
  is consistent for some $i\in\omega$ (as otherwise $\varphi^{k}$
  would have $\TP_{2}$ by the mutual indiscernibility of rows), thus
  for $i=0$ (as the sequence of rows is indiscernible). Unwinding,
  we conclude that $\left\{ \varphi\left(x,a_{ij}\right)\right\} _{i<k,j\in\omega}$
  is consistent.
\end{proof}

\begin{lem}
  \label{lem: Consistent power diagonal}
  Assume that $T$ is $\NTP_2$ and let $\mathbf{A}=\left(a_{ij}\right)_{i,j\in\omega}$
  be an indiscernible array and assume that the diagonal $\left\{ \varphi\left(x,a_{ii}\right)\right\} _{i\in\omega}$
  is consistent. Then for any $k\in\omega$, if $\mathbf{A}^{k}=\left(b{}_{ij}\right)_{i,j\in\omega}$
  then the diagonal $\left\{ \varphi^{k}\left(x,b_{ii}\right)\right\} _{i\in\omega}$
  is consistent.
\end{lem}
\begin{proof}
  By compactness we can extend our array $\mathbf{A}$ to $\left(a_{ij}\right)_{i\in\omega\times\omega,j\in\omega}$
  and let $b_{ij}=a_{i\times\omega+j,i}$.

  It then follows that $\left(b_{ij}\right)_{i,j\in\omega}$ is a strongly
  indiscernible array and that $\left\{ \varphi\left(x,b_{i0}\right)\right\} _{i\in\omega}$
  is consistent. But then $\left\{ \varphi\left(x,b_{ij}\right)\right\} _{i,j\in\omega}$
  is consistent by \autoref{lem: very indiscernible array is consistent}
  , and we can conclude by indiscernibility of $\mathbf{A}$.

  \begin{figure}[h]
    \includegraphics[scale=0.5]{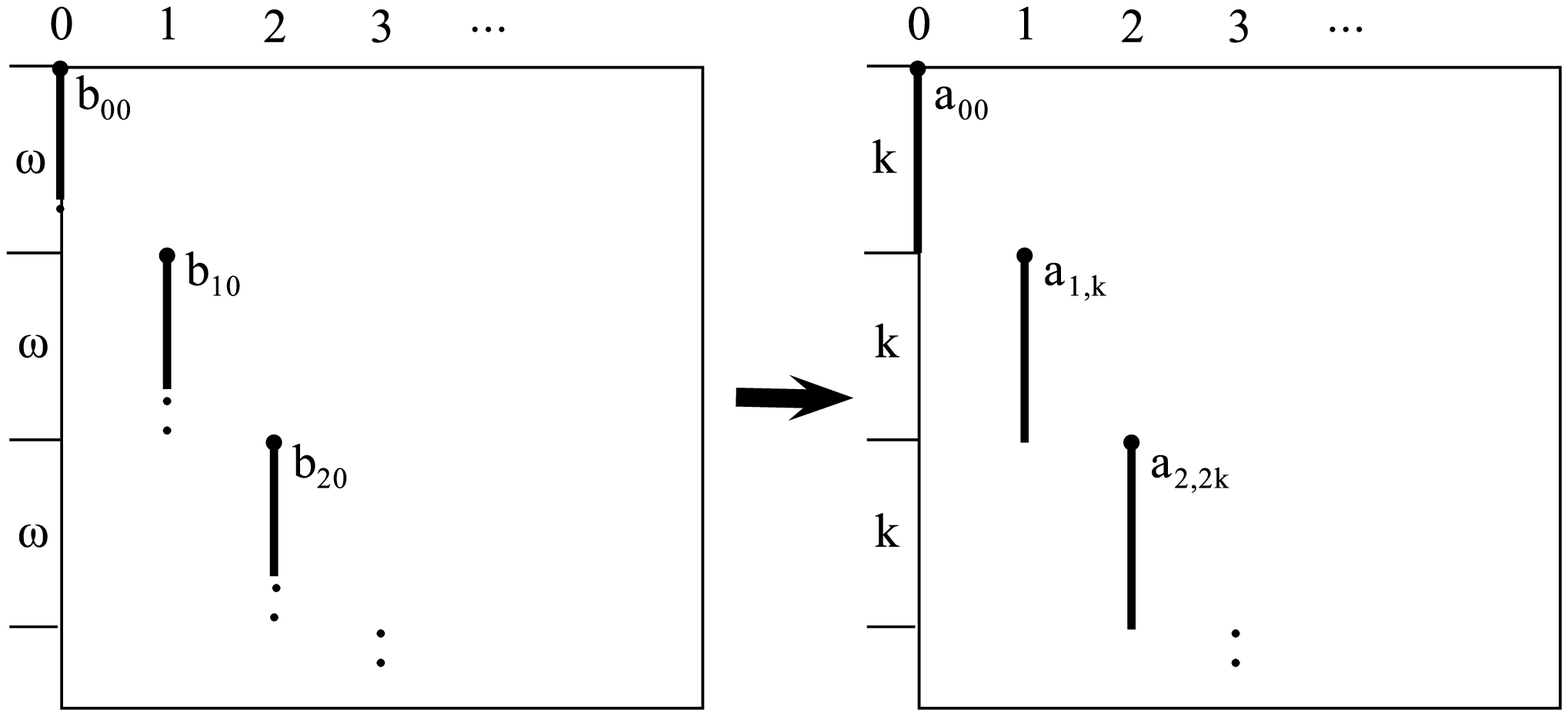}
  \end{figure}
\end{proof}

\begin{prp}
  \label{prop: consistent diagonal implies consistent array} Assume
  $T$ is $\NTP_2$. If $\left(a_{ij}\right)_{i,j\in\omega}$ is an
  indiscernible array and the diagonal $\left\{ \varphi(x,a_{ii})\right\} _{i\in\omega}$
  is consistent, then the whole array $\left\{ \varphi(x,a_{ij})\right\} _{i,j\in\omega}$
  is consistent. Moreover, this property characterizes $\NTP_2$.
\end{prp}
\begin{proof}
  Let $\kappa\in\omega$ be arbitrary. Let $\mathbf{A}^{k}=\left(b_{ij}\right)_{i,j\in\omega}$,
  then its diagonal $\left\{ \varphi^{k}\left(x,b_{ii}\right)\right\} _{i\in\omega}$
  is consistent by \autoref{lem: Consistent power diagonal}. As $\mathbf{B}=\left(\mathbf{A}^{k}\right)^{T}$
  has the same diagonal, using \autoref{lem: Consistent power diagonal}
  again we conclude that if $\mathbf{B}^{k}=\left(c_{ij}\right)_{i,j\in\omega}$,
  then its diagonal $\left\{ \varphi^{k^{2}}\left(x,c_{ii}\right)\right\} _{i\in\omega}$
  is consistent. In particular $\left\{ \varphi\left(x,a_{ij}\right)\right\} _{i,j<k}$
  is consistent. Conclude by compactness.

  \begin{figure}[h]
    \includegraphics[scale=0.5]{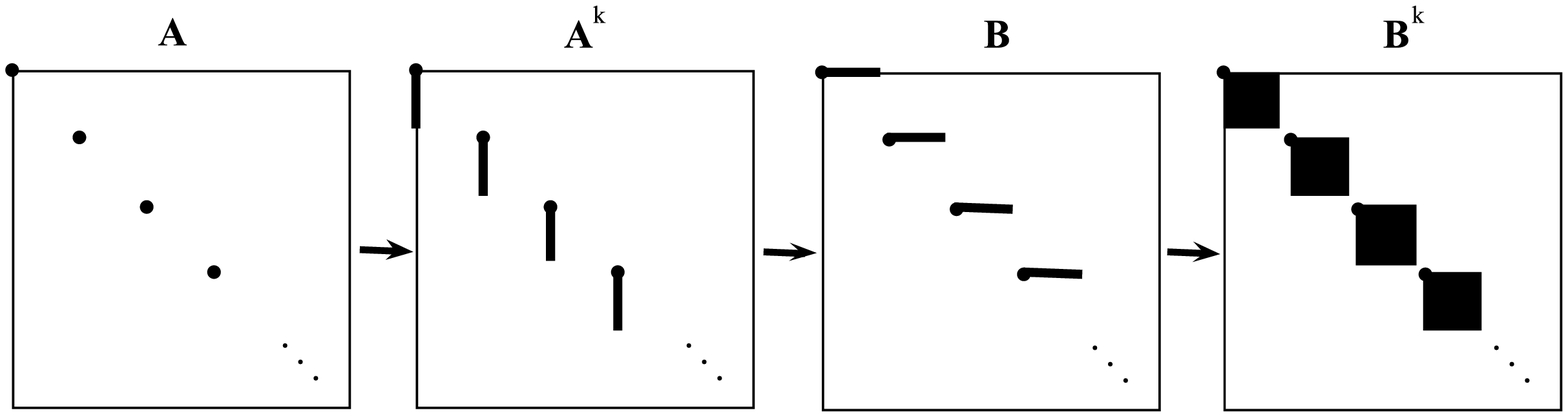}
  \end{figure}

  ``Moreover'' follows from the fact that if $T$ has $\TP_{2}$,
  then there is a strongly indiscernible array witnessing this.
\end{proof}

\begin{cor}
  \label{cor: dividing =00003D array dividing} Let $T$ be $\NTP_2$.
  Then $\varphi(x,a)$ divides over $A$ if and only if it array-divides
  over $A$.
\end{cor}
\begin{proof}
  If $\left(a_{ij}\right)_{i,j\in\omega}$ is an $A$-indiscernible
  array with $a_{00}=a$, then $\left\{ \varphi(x,a_{ii})\right\} _{i\in\omega}$
  is consistent since $\left(a_{ii}\right)_{i\in\omega}$ is indiscernible
  over $A$ and $\varphi(x,a)$ does not divide over $A$, apply \autoref{prop: consistent diagonal implies consistent array}.
\end{proof}
\begin{rmk}
  Array dividing was apparently first considered for the purposes of
  classification of Zariski geometries in \cite{Hrushovski-Zilber:ZariskiGeometries}.
  Kim \cite{Kim:PhD} proved that in simple theories dividing equals
  array dividing. Later the first author used it to develop the basics
  of simplicity theory in the context of compact abstract theories \cite{BenYaacov:SimplicityInCats},
  and Adler used it in his presentation of thorn-forking in \cite{Adler:ThornForking}.
\end{rmk}

\section{\label{sec:Chain-condition} The chain condition}

\subsection{The chain condition}
\begin{dfn}
  We say that forking in $T$ satisfies the \emph{chain condition} over
  $A$ if whenever $I=\left(a_{i}\right)_{i\in\omega}$ is an indiscernible
  sequence over $A$ and $\varphi(x,a_{0})$ does not fork over $A$,
  then $\varphi(x,a_{0})\land\varphi(x,a_{1})$ does not fork over $A$.
  It then follows that $\left\{ \varphi(x,a_{i})\right\} _{i\in\omega}$
  does not fork over $A$.
\end{dfn}
\begin{lem}
  \label{lem: equivalents of the chain condition} The following are
  equivalent for any theory $T$ and a set $A$:
  \begin{enumerate}
  \item Forking in $T$ satisfies the chain condition over $A$.
  \item
    \label{item:ChainConditionExplicit}
    Let $\kappa = (2^{|T|+|A|})^+$.
    Then for every $p(x)\in S(A)$, whenever $\left(p(x)\cup\left\{ \varphi_{i}(x,a_{i})\right\} \right)_{i<\kappa}$ is a family of partial types non-forking over $A$, there are $i<j<\kappa$ such that $p(x)\cup\left\{ \varphi_{i}(x,a_{i})\right\} \cup\left\{ \varphi_{j}(x,a_{j})\right\} $ does not fork over $A$.
  \item
    \label{item:ChainCondition}
    The previous item holds for \emph{some} $\kappa$.
    In other words, there are no anti-chains of unbounded size in the partial order of non-forking types over $A$.
  \item If $b\ind_{A}a_{0}$ and $I=\left(a_{i}\right)_{i\in\omega}$ is indiscernible
    over $A$, then there is $I'\equiv_{Aa_{0}}I$, indiscernible over
    $Ab$ and such that $b\ind_{A}I'$.
  \end{enumerate}
\end{lem}
\begin{proof}
  \begin{cycprf}
  \item Follows from the fact that in every set $S$ with
    elements of size $\lambda$, if $\left|S\right|>2^{\lambda+\left|T\right|}$
    then some two different elements appear in an indiscernible sequence
    (see e.g. \cite[Proposition 3.3]{Casanovas:Dividing}).
  \item Obvious.
  \item
    We may assume that $I$ is of length $\kappa$, long
    enough. Let $p(x,a_{0})=\tp(b/a_{0}A)$. It follows from \autoref{item:ChainCondition} by compactness
    that $\bigcup_{i<\kappa}p(x,a_{i})$ does not fork over $A$. Then
    there is $b'$ realizing it, such that in addition $b'\ind_{A}I$.
    By Ramsey, automorphism and compactness we find an $I'$ as wanted.
  \item[\impfirst]
    Assume that the chain condition fails, let $I$ and $\varphi(x,y)$
    witness this, so $\varphi(x,a_{0})\land\varphi(x,a_{1})$ forks over
    $A$. Let $b\vDash\varphi(x,a_{0})\land\varphi(x,a_{1})$. It is
    clearly not possible to find $I'$ as in (4).
  \end{cycprf}
\end{proof}
\begin{rmk}
  The term ``chain condition'' refers to \autoref{lem: equivalents of the chain condition}\autoref{item:ChainCondition}
  interpreted as saying that there are no antichains of unbounded size
  in the partial order of non-forking formulas (ordered by implication).
  The chain condition was introduced and proved by Shelah with respect
  to weak dividing, rather than dividing, for simple theories in the
  form of \autoref{item:ChainConditionExplicit} in \cite{Shelah:SimpleUnstableTheories}. Later \cite[Theorem~4.9]{Grossberg-Iovino-Lessman:SimplePrimer}
  presented a proof due to Shelah of the chain condition with respect
  to dividing for simple theories using the independence theorem, again
  in the form of \autoref{item:ChainConditionExplicit}. The chain condition as defined here was proved
  for simple theories by Kim \cite{Kim:PhD}. It was further studied
  by Dolich \cite{Dolich:WeakDividing}, Lessmann \cite{Lessmann:CountingPartialTypes}, Casanovas
  \cite{Casanovas:Dividing} and Adler \cite{Adler:Chain} establishing the equivalence
  of the first three forms. In the case of $\NIP$ theories, the chain condition
  follows immediately from the fact that non-forking is equivalent to
  Lascar-invariance (see \autoref{lem: invariance satisfies chain condition}).
\end{rmk}
Of course, the chain condition need not hold in general.
\begin{exm}
  Let $T$ be the model completion of the theory of triangle-free graphs.
  It eliminates quantifiers. Let $M\vDash T$ and let $\left(a_{i}\right)_{i\in\omega}$
  be an $M$-indiscernible sequence such that $\vDash\neg Ra_{i}b$
  for any $i$ and $b\in M$. Notice that by indiscernibility $\vDash\neg Ra_{i}a_{j}$
  for $i\neq j$. It is easy to see that $Rxa_{0}$ does not divide
  over $M$. On the other hand, $Rxa_{0}\land Rxa_{1}$ divides over
  $M$.
\end{exm}

\subsection{NTP$_2$ implies the chain condition.\protect \\
}

We will need some facts about forking and dividing in $\NTP_2$
theories established in \cite{Chernikov-Kaplan:ForkingDividingNTP2}. Recall that a set $C$ is
an \emph{extension base} if every type in $S(C)$ does not fork over
$C$.
\begin{dfn}
  We say that $\left(a_{i}\right)_{i\in\kappa}$ is a \emph{universal
    Morley sequence} in $p(x)\in S(A)$ when:
  \begin{itemize}
  \item it is indiscernible over $A$ with $a_{i}\vDash p(x)$
  \item for any $\varphi(x,y)\in L(A)$, if $\varphi(x,a_{0})$ divides over
    , then $\left\{ \varphi(x,a_{i})\right\} _{i\in\kappa}$ is inconsistent.
  \end{itemize}
\end{dfn}
\begin{fct}
  \label{fct: Forking and dividing in NTP2} \cite{Chernikov-Kaplan:ForkingDividingNTP2} Assume that $T$
  is $\NTP_2$.
  \begin{enumerate}
  \item Let $M$ be a model. Then for every $p(x)\in S(M)$, there is a universal
    Morley sequence in it.
  \item Let $C$ be an extension base. Then $\varphi(x,a)$ divides over $C$
    if and only if $\varphi(x,a)$ forks over $C$.
  \end{enumerate}
\end{fct}
First we observe that the chain condition always implies equality
of dividing and array dividing:
\begin{prp}
  \label{prop: Chain condition} If $T$ satisfies the chain condition
  over $C$, and forking equals dividing over $C$, then $\varphi(x,a)$ divides over $C$ if and only if it
  array-divides over $C$.
\end{prp}
\begin{proof}
  Assume that $\varphi(x,a)$ does not divide over $C$. Let $\left(a_{ij}\right)_{i,j\in\omega}$
  be a $C$-indiscernible array and $a_{00}=a$. It follows by the chain
  condition and compactness that $\left\{ \varphi\left(x,a_{i0}\right)\right\} _{i\in\omega}$
  does not divide over $C$. But as $\left(\left(a_{ij}\right)_{i\in\omega}\right)_{j\in\omega}$
  is also a $C$-indiscernible sequence, applying the chain condition
  and compactness again we conclude that $\left\{ \varphi\left(x,a_{ij}\right)\right\} _{i,j\in\omega}$
  does not divide over $C$, so in particular it is consistent.
\end{proof}
And in the presence of universal Morley sequences witnessing dividing,
the converse holds:
\begin{prp}
  \label{prop: Chain condition over models} Let $T$ be $\NTP_2$
  and $M\vDash T$. Then forking satisfies the chain condition over
  $M$.
\end{prp}
\begin{proof}
  Let $\kappa$ be very large compared to $\left|M\right|$, assume
  that $\bar{a}_{0}=\left(a_{0i}\right)_{i\in\kappa}$ is indiscernible
  over $M$, $\varphi(x,a_{00})$ does not divide over $M$, but $\varphi(x,a_{00})\land\varphi(x,a_{01})$
  does. By \autoref{fct: Forking and dividing in NTP2}, let $\left(\bar{a}_{i}\right)_{i\in\omega}$
  be a universal Morley sequence in $\tp(\bar{a}_{0}/M)$. By the universality
  and indiscernibility of $\bar{a}_{0}$, $\left\{ \varphi(x,a_{ij_{1}})\land\varphi(x,a_{ij_{2}})\right\} _{i\in\omega}$
  is inconsistent for any $j_{1}\neq j_{2}$. We can extract an $M$-indiscernible
  sequence $\left(\left(a_{ij}'\right)_{i\in\omega}\right)_{j\in\omega}$
  from $\left(\left(a_{ij}\right)_{i\in\omega}\right)_{j\in\kappa}$,
  such that type of every finite subsequence over $M$ is already present
  in the original sequence. It follows that $\left(a_{ij}'\right)_{i,j\in\omega}$
  is an $M$-indiscernible array and that $\left\{ \varphi(x,a_{ij}')\right\} _{i,j\in\omega}$
  is inconsistent, thus $\varphi(x,a_{00})$ array-divides over $M$,
  thus divides over $M$ by \autoref{cor: dividing =00003D array dividing}
  --- a contradiction.
\end{proof}
\begin{thm}
  \label{thm: NTP2 implies Chain Condition} If $T$ is $\NTP_2$,
  then it satisfies the chain condition over extension bases.
\end{thm}
\begin{proof}
  Let $C$ be an extension base and $\bar{a}=\left(a_{i}\right)_{i\in\omega}$
  be a $C$-indiscernible sequence. As $C$ is an extension base, we
  can find $M\supseteq C$ such that $M\ind_{C}\bar{a}$. It follows
  that for any $n\in\omega$, $\bigwedge_{i<n}\varphi(x,a_{i})$ divides
  over $C$ if and only if it divides over $M$. It follows from \autoref{prop: Chain condition over models} that if $\varphi(x,a_{0})$
  does not divide over $C$, then $\left\{ \varphi(x,a_{i})\right\} _{i\in\omega}$
  does not divide over $C$.
\end{proof}
\begin{cor}
  If $T$ is $\NTP_2$, $A$ is an extension base, $\left(a_{ij}\right)_{i,j\in\omega}$
  is an $A$-indiscernible array, and $\varphi\left(x,a_{00}\right)$
  does not divide over $A$, then $\left\{ \varphi\left(x,a_{ij}\right)\right\} _{i,j\in\omega}$
  does not divide over $A$.
\end{cor}

\subsection{\label{sec: Chain condition does not imply NTP2} The chain condition
  does not imply $\NTP_2$}
\begin{lem}
  \label{lem: invariance satisfies chain condition} Let $T$ be a theory
  satisfying:
  \begin{itemize}
  \item For every set $A$ and a global type $p(x)$, it does not fork over
    $A$ if and only if it is Lascar-invariant over $A$.
  \end{itemize}
  Then $T$ satisfies the chain condition.
\end{lem}
\begin{proof}
  Let $\bar{a}=\left(a_{i}\right)_{i\in\omega}$ be an $A$-indiscernible
  sequence and assume that $\varphi(x,a_{0})$ does not fork over $A$.
  Then there is a global type $p(x)$ containing $\varphi(x,a_{0})$
  and non-forking over $A$, thus Lascar-invariant over $A$. Taking
  $c\vDash p|_{\bar{a}A}$, it follows by Lascar-invariance that $c\vDash\left\{ \varphi(x,a_{i})\right\} _{i\in\omega}$.
\end{proof}
In \cite[Section 5.3]{Chernikov-Kaplan-Shelah:NonForkingSpectra} the following example is constructed:
\begin{fct}
  There is a theory $T$ such that:
  \begin{enumerate}
  \item $T$ has $\TP_{2}$.
  \item A global type does not fork over a small set $A$ if and only if it
    is finitely satisfiable in $A$ (therefore, if and only if it is Lascar-invariant
    over $A$).
  \end{enumerate}
\end{fct}
It follows from \autoref{lem: invariance satisfies chain condition}
that this $T$ satisfies the chain condition.

\section{\label{sec: Weak-independence-theorem} The independence theorem and Lascar types}
\begin{dfn}
  \label{def: Lascar strong type} As usual, we write $a\equiv_{C}^{\lstp}b$
  to denote that $a$ and $b$ have the same \emph{Lascar type} over
  $C$. That is, if any of the following equivalent properties holds:
  \begin{enumerate}
  \item $a$ and $b$ are equivalent under every $C$-invariant equivalence
    relation with a bounded number of classes.
  \item There are $n\in\omega$ and $a=a_{0},...,a_{n}=b$ such that $a_{i},a_{i+1}$
    start a $C$-indiscernible sequence for each $i<n$.
  \end{enumerate}
  We let $d_{C}\left(a,b\right)$ be the \emph{Lascar distance}, that
  is the smallest $n$ as in (2) or $\infty$ if it does not exist.
\end{dfn}
Now we will use the chain condition in order to deduce a independence
theorem over an extension base.
\begin{lem}
  \label{lem: indisc triangles} Assume that $d_{A}\left(b,b'\right)=1$
  and $a\ind_{Ab}b'$. Then there exists a sequence $\left(a{}_{i}b{}_{i}\right)_{i\in\omega}$
  indiscernible over $A$ and such that $a_{0}b_{0}b_{1}=abb'$.
\end{lem}
\begin{proof}
  Standard.
\end{proof}
\begin{thm}
  \label{thm: Weak Independence Theorem} Let $T$ be $\NTP_2$ and
  $A$ an extension base. Assume that $c\ind_{A}ab$, $a\ind_{A}bb'$
  and $b\equiv_{A}^{\lstp}b'$. Then there is $c'$ such that $c'\ind_{A}ab'$,
  $c'a\equiv_{A}ca$, $c'b'\equiv_{A}cb$.
\end{thm}
\begin{proof}
  Let us first consider the case $d_{A}\left(b,b'\right)=1$. Since
  $a\ind_{Ab}b'$, by \autoref{lem: indisc triangles} we can find
  $\left(a{}_{i}b{}_{i}\right)_{i\in\omega}$ indiscernible over $A$
  and such that $a_{0}b_{0}b_{1}=abb'$. As $c\ind_{A}a{}_{0}b{}_{0}$,
  it follows by the chain condition that there exists $c'\equiv_{Aa_{0}b{}_{0}}c$
  such that $c'\ind_{A}\left(a_{i}b_{i}\right)_{i\in\omega}$ and $\left(a_{i}b_{i}\right)_{i\in\omega}$
  is indiscernible over $c'A$. In particular $c'\ind_{A}ab'$, $c'a\equiv_{A}ca$
  and $c'b'\equiv_{A}c'b\equiv_{A}cb$, as desired.

  For the general case, assume that $d_{A}\left(b,b'\right)\leq n$,
  namely that there are $b_{0},...,b_{n}$ be such that $b_{i}b_{i+1}$
  start an $A$-indiscernible sequence for all $i<n$ and $b_{0}=b$,
  $b_{n}=b'$. We may assume that $a\ind_{A}b_{0}...b_{n}$.

  By induction on $i\leq n$ we choose $c_{i}$ such that:
  \begin{enumerate}
  \item $c_{i}\ind_{A}ab_{i}$,
  \item $c_{i}a\equiv_{A}ca$,
  \item $c_{i}b_{i}\equiv_{A}cb_{0}$.
  \end{enumerate}
  Let $c_{0}=c$, it satisfies (1)--(3) by hypothesis. Given $c_{i}$,
  by the Lascar distance 1 case there is some $c_{i+1}\ind_{A}ab_{i+1}$
  such that $c_{i+1}a\equiv_{A}c_{i}a\equiv_{A}ca$ and $c_{i+1}b_{i+1}\equiv_{A}c_{i}b_{i}\equiv_{A}cb_{0}$
  (by the inductive assumption).

  It follows that $c'=c_{n}$ is as wanted.
\end{proof}
\begin{rmk}
  For simplicity of notation, let us work over $A=\emptyset$.
\begin{enumerate}
  \item It is easy to see that the usual statement of the independence theorem for simple theories implies
  this one. Indeed, let $c_{1}$ be such that $c_{1}b'\equiv^{\lstp}cb$.
  Then $c_{1}\ind b'$, $c\ind a$, $a\ind b'$ and $c_{1}\equiv^{\lstp}c$.
  By the independence theorem we find $c'$ such that $c'\ind ab'$,
  $c'a\equiv ca$ and $c'b'\equiv c_{1}b'\equiv cb$.

  \item Conversely, in a simple theory, the usual independence theorem follows from
  ours by a direct forking calculus argument. Indeed, assume
  that we are given $d_{1}\ind e_{1}$, $d_{2}\ind e_{2}$, $d_{1}\equiv^{\lstp}d_{2}$
  and $e_{1}\ind e_{2}$. Using symmetry and \autoref{lem: extending Lstp}
  we find $e_{1}'d_{2}'$ such that $e_{1}'d_{2}'\ind e_{1}e_{2}$ and
  $e_{1}'d_{2}'\equiv^{\lstp}e_{1}d_{1}$. It is easy to check that
  all the assumptions of \autoref{thm: Weak Independence Theorem} are satisfied
  with $c=d_{2}'$, $b=e_{1}'$, $a=e_{2}$ and $b'=e_{1}$. Applying
  it we find some $d$ such that $d\ind e_{1}e_{2}$, $de_{1}\equiv d_{2}'e_{1}'\equiv d_{1}e_{1}$
  and $de_{2}\equiv d_{2}e_{2}$.

\end{enumerate}
\end{rmk}
We observe that the chain condition means precisely that the ideal
of forking formulas is S1, in the terminology of Hrushovski \cite{Hrushovski:ApproximateSubgroups}.
Combining \autoref{prop: Chain condition} with \cite[Theorem~2.18]{Hrushovski:ApproximateSubgroups}
we can slightly relax the assumption on the independence between the
elements, at the price of assuming that some type has a global invariant
extension:
\begin{prp}
  Let $T$ be $\NTP_2$ and $A$ an extension base. Assume that $c\ind_{A}ab$,
  $b\ind_{A}a$, $b'\ind_{A}a$, $b\equiv_{A}b'$ and $\tp\left(a/A\right)$
  extends to a global $A$-invariant type. Then there exists $c'\ind_{A}ab'$
  and $c'b'\equiv_{A}cb$, $c'a\equiv_{A}ca$.
\end{prp}
Using \autoref{thm: Weak Independence Theorem}, we can show that in $\NTP_2$
theories Lascar types coincide with Kim-Pillay strong types over extension
bases.
\begin{cor}
  \label{cor: lascar type is compact} Assume that $T$ is $\NTP_2$
  and $A$ is an extension base. Then $d\equiv_{A}^{\lstp}e$ if and
  only if $d_{A}(d,e)\leq3$.
\end{cor}
\begin{proof}
  Let $d\equiv_{A}^{\lstp}e$ and let $\left(d_{i}\right)_{i\in\omega}$
  be a Morley sequence over $A$ starting with $d=d_{0}$. As $d_{\geq1}\ind_{A}d_{0}$,
  we may assume that $d_{\geq1}\ind_{A}d_{0}e$.

  We have:
  \begin{itemize}
  \item $d_{>1}\ind_{A}d_{0}d_{1}$
  \item $d_{1}\ind_{A}d_{0}e$
  \item $d_{0}\equiv_{A}^{\lstp}e$
  \end{itemize}
  Applying \autoref{thm: Weak Independence Theorem} (with $a=d_{1}$, $b=d_{0}$,
  $b'=e$ and $c=d_{>1}$) we get some $d'_{>1}$ such that $d_{1}d'_{>1}\equiv_{A}d_{1}d_{>1}$
  (thus $d_{1}+d'_{>1}$ is an $A$-indiscernible sequence) and $ed_{>1}'\equiv_{A}d_{0}d_{>1}$
  (thus $e+d'_{>1}$ is an $A$-indiscernible sequence). It follows
  that $d_{A}(d,e)\leq3$ along the sequence $d,d_{1},d_{2}',e$.
\end{proof}
\begin{rmk}
  Consider the standard example \cite[Section~4]{Casanovas-Lascar-Pillay-Ziegler:GaloisGroups}
  showing that the Lascar distance can be exactly $n$ for any $n\in\omega$.
  It is easy to see that this theory is $\NIP$, as it is interpretable
  in the real closed field. However, $\emptyset$ is not an extension
  base.
\end{rmk}
It is known that both in simple theories (for arbitrary $A$) and
in $\NIP$ theories (for $A$ an extension base), $a\equiv_{A}b$
implies that $d_{A}\left(a,b\right)\leq2$ (\cite[Corollary~2.10(i)]{Hrushovski-Pillay:NIPInvariantMeasures}),
while our argument only gives an upper bound of $3$. Thus it is natural
to ask:
\begin{qst}
  Is there an $\NTP_2$ theory $T$, an extension base $A$ and tuples
  $a,b$ such that $d_{A}\left(a,b\right)=3$?
\end{qst}
\begin{dfn}
  Let $a\equiv'_{A}b$ be the transitive closure of the relation ``$a,b$
  start a Morley sequence over $A$, or $b,a$ starts a Morley sequence
  over $A$''. This is an $A$-invariant equivalence relation refining
  $\equiv_{A}^{\lstp}$.
\end{dfn}
The proof of \autoref{cor: lascar type is compact} demonstrates
in particular that if $A$ is an extension base in an $\NTP_2$
theory, then $a\equiv_{A}^{\lstp}b$ if and only if $a\equiv_{A}'b$.
We show that in fact this holds in a much more general setting.

Let $T$ be an arbitrary theory. We call a type $p\left(x\right)\in S\left(A\right)$
\emph{extensible} if it has a global extension non-forking over $A$,
equivalently if it does not fork over $A$ (thus $A$ is an extension
base if and only if every type over it is extensible).
\begin{lem}
  \label{lem: extending Lstp} Let $\tp\left(a/A\right)$ be extensible.
  Then for any $b$ there is some $a'$ such that $a'\equiv_{A}'a$
  and $a'\ind_{A}b$.
\end{lem}
\begin{proof}
  Let $\left(a_{i}\right)_{i\in\omega}$ be a Morley sequence over $A$
  starting with $a_{0}$. It follows that $a_{\geq1}\ind_{A}a_{0}$.
  Then there is $a_{\geq1}'\ind_{A}a_{0}b$ and such that $a_{\geq1}\equiv_{a_{0}A}a_{\geq1}'$.
  In particular $a_{0}+a_{\geq1}'$ is still a Morley sequence over
  $A$, thus $a_{1}'\equiv_{A}'a_{0}$, and $a_{1}'\ind_{A}b$ as wanted.
\end{proof}
\begin{prp}
  Let $p$ be an extensible type. Then $a\equiv_{A}^{\lstp}b$ if and
  only if $a\equiv'_{A}b$, for any $a,b\vDash p\left(x\right)$.
\end{prp}
\begin{proof}
  By \autoref{def: Lascar strong type}(1) it is enough to show
  that $\equiv'_{A}$ has boundedly many classes on the set of realizations
  of $p$.

  Assume not, and let $\kappa$ be large enough. We will choose $\equiv'$-inequivalent
  $\left(a_{i}\right)_{i\in\kappa}$ such that in addition $a_{i}\ind_{A}a_{<i}$.
  Suppose we have chosen $a_{<j}$ and let us choose $a_{j}$. Let $b\vDash p$
  be $\equiv_{A}'$-inequivalent to $a_{i}$ for all $i<j$. By \autoref{lem: extending Lstp}, there exists $a_{j}\equiv'_{A}b$ such
  that $a_{j}\ind_{A}a_{<j}$. In particular $a_{j}\not\equiv'_{A}a_{i}$
  for all $i<j$ as desired.

  With $\kappa$ sufficiently large, we may extract an $A$-indiscernible
  sequence $\bar{b}=\left(b_{i}\right)_{i\in\omega}$ from $\left(a_{i}\right)_{i\in\kappa}$
  --- a contradiction, as then $\bar{b}$ is a Morley sequence over
  $A$ but $b_{i}\not\equiv'_{A}b_{j}$ for any $i\neq j$.
\end{proof}

\section{\label{sec: Dividing-order} The dividing order}

In this section we suggest a generalization of the fundamental order
of Poizat \cite{Poizat:Cours} in the context of $\NTP_2$ theories.
For simplicity of notation, we only consider $1$-types, but everything
we do holds for $n$-types just as well.

Given a partial type $r\left(x\right)$ over $A$, we let $S^{\EM,r}(A)$
be the set of Ehrenfeucht-Mostowski types of $A$-indiscernible sequences
in $r(x)$. We will omit $A$ when $A=\emptyset$ and omit $r$ when
it is ``$x=x$''.
\begin{dfn}
  Given $p\in S^{\EM}\left(A\right)$, let $\cl^{\div}(p)$ be the
  set of all $\varphi(x,y)\in L\left(A\right)$ such that for some (any)
  infinite $A$-indiscernible sequences $\bar{a}\vDash p$, the set $\left\{ \varphi(a_{i},y)\right\} _{i\in\omega}$
  is consistent. For $p,q\in S^{\EM}\left(A\right)$, we say that
  $p\sim_{A}^{\div}q$ ( respectively, $p\leq_{A}^{\div}q$) if $\cl^{\div}(p)=\cl^{\div}(q)$
  (respectively, $\cl^{\div}(p)\supseteq\cl^{\div}(q)$). We obtain
  a partial order $\left(S^{\EM}(A)/\sim_{A}^{\div},\leq_{A}^{\div}\right)$.
\end{dfn}
\begin{prp}
  Let $T$ be stable. Then $p\sim^{\div}q$ if and only if $p=q$, and
  $\left(S^{\EM},\leq^{\div}\right)$ is isomorphic to the fundamental
  order of $T$.
\end{prp}
\begin{proof}
  For a type $p$ over a model $M$ we let $\cl(p)$ denote its fundamental
  class, namely the set of formulas $\varphi(x,y)$ such that there
  exists an instance $\varphi(x,b)\in p(x)$. We denote the fundamental
  order of $T$ by $\left(S/\sim^{\fund},\leq^{\fund}\right)$ where
  $S$ is the set of all types over all models of $T$, $p\leq^{\fund}q$
  if $\cl(p)\supseteq\cl(q)$ and $\sim^{\fund}$ is the corresponding
  equivalence relation. Given $p\in S\left(M\right)$, let $p^{\left(\omega\right)}\in S_{\omega}\left(M\right)$
  be the type of its Morley sequence over $M$. By stability $p^{\left(\omega\right)}$
  is determined by $p$. Let $p^{\EM}$ be the Ehrenfeucht-Mostowski
  type over the empty set of $\bar{a}\vDash p^{\left(\omega\right)}|_{M}$.
  Let $f:\, S\to S^{\EM}$, $f:\, p\mapsto p^{\EM}$.
  \begin{enumerate}
  \item \label{enu: lemma for fund order} Given $p\in S\left(M\right)$,
    let $\bar{a}\vDash p^{\left(\omega\right)}$, and let us show that
    $\varphi(x,y)\in\cl(p)$ if and only if $\left\{ \varphi\left(a_{i},y\right)\right\} _{i\in\omega}$
    is consistent. Indeed, by stability, either condition is equivalent
    to:$\varphi(a_{0},y)$ does not divide over $M$. In other words,
    $\cl(p)=\cl^{\div}(f(p))$, so $p\leq^{\fund}q\,\Leftrightarrow\, f\left(p\right)\leq^{\div}f\left(q\right)$.
  \item We show that $f$ is onto\textbf{.} Let $P\in S^{\EM}$ be arbitrary,
    and let $\left(a_{i}\right)_{i\in2\omega}$ be an indiscernible sequence
    with $P$ as its EM type. Let $M$ be a model containing $I=\left(a_{i}\right)_{i\in\omega}$,
    such that $J=\left(a_{\omega+i}\right)_{i\in\omega}$ is indiscernible
    over $M$. Then $J$ is a Morley sequence in $p\left(x\right)=\tp\left(a_{\omega}/M\right)$,
    and $f\left(p\right)=P$, as wanted.
  \item To conclude, let $P,Q\in S^{\EM}$, $P\sim^{\div}Q$, and let
    us show that they are equal. Let $p\in S(M)$ and $q\in S(N)$ be
    sent by $f$ to $P$ and $Q$, respectively. Since $Th(M)\subseteq\cl^{\div}(P)$
    and similarly for $N,Q$, we have $M\equiv N$. Taking non-forking
    extensions of $p,q$, we may therefore assume that $M=N$ is a monster
    model. Since $\cl(p)=\cl(q)$, the types of (the parameters of) their
    definitions are the same, so there exists an automorphism sending
    one definition to the other, and therefore sending $p\mapsto q$.
    Since $f(p)$ does not involve any parameters, it follows that $P=f(p)=f(q)=Q$.
  \end{enumerate}
\end{proof}
\begin{rmk}
  A couple of remarks on the existence of the greatest element in the
  dividing order in $\NTP_2$ theories.
  \begin{enumerate}
  \item Given a type $r(x_{1},x_{2})\in S(A)$, assume that $p\left(\left(x_{1j},x_{2j}\right)_{j\in\omega}\right)$
    is the greatest element in $S^{\EM,r}(A)$ (modulo $\sim_{A}^{\div}$).
    Then for $i=1,2$, $p_{i}\left(\left(x_{ij}\right)_{j\in\omega}\right)=p|_{\left(x_{ij}\right)_{j\in\omega}}$
    is the greatest element in $S^{\EM,r_{i}}(A)$ with $r_{i}=r|_{x_{i}}$.
  \item If for every $r\in S(A)$ there is a $\leq^{\div}$-greatest element
    in $S^{\EM,r}(A)$, then a formula $\varphi(x,a)$ forks over
    $A$ if and only if it divides over $A$.
  \item If $T$ is $\NTP_2$ then for every extension base $A$ and $r\in S(A)$
    there is a $\leq^{\div}$-greatest element in $S^{\EM,r}(A)$.
  \end{enumerate}
\end{rmk}
\begin{proof}

  \begin{enumerate}
  \item Clear as e.g. given an $A$-indiscernible sequence $\left(a_{1j}\right)_{j\in\omega}$
    in $r_{1}(x_{1})$, by compactness and Ramsey we can find $\left(a_{2j}\right)_{j\in\omega}$
    such that $\left(a_{1j}a_{2j}\right)_{j\in\omega}$ is an $A$-indiscernible
    sequence in $r(x_{1},x_{2})$.
  \item Assume that $\varphi(x,a)\vdash\bigvee_{i<k}\varphi_{i}(x,a_{i})$
    and $\varphi_{i}(x,a_{i})$ divides over $A$ for each $i<k$. Let
    $r(xx_{0}\ldots x_{k-1})=\tp(aa_{0}\ldots a_{k-1}/A)$, let $p(\bar{x}\bar{x}_{0}\ldots\bar{x}_{k-1})$
    be the greatest element in $S^{\EM,r}(A)$ and let $\left(a_{j}a_{0j}\ldots a_{(k-1)j}\right)_{j\in\omega}$
    realize it. As $\left\{ \varphi(x,a_{j})\right\} _{j\in\omega}$ is
    consistent, it follows that $\left\{ \varphi_{i}(x,a_{ij})\right\} _{j\in\omega}$
    is consistent for some $i<k$ --- contradicting the assumption that
    $\varphi_{i}(x,a_{i})$ divides by (i).
  \item Let $a\vDash r$. As $A$ is an extension base, let $M\supseteq A$
    be a model such that $M\ind_{A}a$. Let $I=\left(a_{i}\right)_{i\in\omega}$
    be a universal Morley sequence in $\tp(a/M)$ which exists by \autoref{fct: Forking and dividing in NTP2}. Then $\tp(I/A)$ is the greatest
    element in $S^{\EM,r}(A)$. Indeed, $\varphi(x,a)$ divides over
    $A$ $\Leftrightarrow$ $\varphi(x,a)$ divides over $M$ $\Leftrightarrow$
    $\left\{ \varphi(x,a_{i})\right\} _{i\in\omega}$ is inconsistent.
  \end{enumerate}
\end{proof}
\begin{dfn}
  For $p,q\in S^{\EM}$, we write $p\leq^{\#}q$ if there is an
  array $\left(a_{ij}\right)_{i,j\in\omega}$ such that:
  \begin{itemize}
  \item $\left(a_{ij}\right)_{j\in\omega}\vDash p$ for each $i\in\omega$,
  \item $\left(a_{if(i)}\right)_{i\in\omega}\vDash q$ for each $f:\,\omega\to\omega$.
  \end{itemize}
\end{dfn}
\begin{prp}
  Let $p,q\in S^{\EM}$.
  \begin{enumerate}
  \item If $p\leq^{\div}q$, then $p\leq^{\#}q$.
  \item If $T$ is $\NTP_2$ and $p\leq^{\#}q$, then $p\leq^{\div}q$.
  \end{enumerate}
\end{prp}
\begin{proof}
 \begin{enumerate}
  \item We show by induction that for each $n\in\omega$ we can find
  $\left(\bar{a}_{i}\right)_{i\in n}$ and $\bar{b}$ such that: $\bar{a}_{i}\vDash p$
  and $a_{0j_{0}}+...+a_{(n-1)j_{n-1}}+\bar{b}\vDash q$ for any $j_{0},\ldots,j_{n-1}\in\omega$.
  Assume we have found $\left(\bar{a}_{i}\right)_{i<n}$ and $\bar{b}$,
  without loss of generality $\bar{b}=\bar{b}'+\bar{b}''=\left(b_{i}'\right)_{i\in\omega}+\left(b_{i}''\right)_{i\in\omega}$.
  Consider the type

  \begin{eqnarray*}
    r(\bar{x}_{0}...\bar{x}_{n-1},y,\bar{z}) & = & \bigcup_{i < n}p(\bar{x}_{i})\cup q(\bar{z})\cup\\
    \cup & \bigcup_{j_{0},...,j_{n-1}\in\omega} & \mbox{"\ensuremath{x_{0j_{0}}+x_{1j_{1}}+...+x_{(n-1)j_{n-1}}+y+\bar{z}} is indiscernible"}
  \end{eqnarray*}

  For every finite $r'\subset r$, $\left\{ r'(\bar{x}_{0}...\bar{x}_{n-1},y_{i},\bar{z})\right\} _{i\in\omega}\cup q(\bar{y})$
  is consistent --- since by the inductive assumption $\vDash r'(\bar{a}_{0}...\bar{a}_{n-1},b_{i}',\bar{b}'')$
  for all $i\in\omega$. Together with $p\leq^{\div}q$ this implies
  that $\left\{ r'(\bar{x}_{0}...\bar{x}_{n-1},y_{i},\bar{z})\right\} _{i\in\omega}\cup p(\bar{y})$
  is consistent. By compactness we find $\bar{a}_{0},...,\bar{a}_{n-1},\bar{a}_{n},\bar{b}$
  realizing it, and they are what we were looking for.

  \item Follows from the definition of $\TP_{2}$.
 \end{enumerate}
\end{proof}
\begin{dfn}
  We write $p\leq^{+}q$%
  \footnote{Note that ``$\#$'' and ``$+$'' are supposed to graphically represent
    the combinatorial configuration which we are using in the definition
    of the order.%
  } if there is $\bar{a}=\left(a_{i}\right)_{i\in\mathbb{Z}}\vDash q$
  and $\bar{b}=\left(b_{i}\right)_{i\in\mathbb{Z}}\vDash p$ such that
  $a_{0}=b_{0}$ and $\bar{b}$ is indiscernible over $\left(a_{i}\right)_{i\neq0}$.
\end{dfn}
\begin{rmk}
  In any theory, $p\leq^{\#}q$ implies $p\leq^{+}q$ (and so $p\leq^{\div}q$
  implies $p\leq^{+}q$).
\end{rmk}
\begin{proof}
  If $p\leq^{\#}q$, then by compactness and Ramsey we can find an array
  $\left(c_{ij}\right)_{i,j\in\mathbb{Z}}$ such that:
  \begin{itemize}
  \item $\bar{c}_{i}$ is indiscernible over $\bar{c}_{\neq i}$,
  \item $\left(\bar{c}_{i}\right)_{i\in\mathbb{Z}}$ is an indiscernible sequence,
  \item $\bar{c}_{i}\vDash p$ for all $i\in\omega$,
  \item $\left(c_{if(i)}\right)_{i\in\omega}\vDash q$ for all $f:\,\omega\to\omega$.
  \end{itemize}
  Then take $\bar{a}=\left(c_{0j}\right)_{j\in\mathbb{Z}}$ and $\bar{b}=\left(c_{i0}\right)_{i\in\mathbb{Z}}$.
\end{proof}
It is much less clear, however, if the converse implication holds.
\begin{dfn}
  We say that $T$ is \emph{resilient}%
  \footnote{The term was suggested by Hans Adler as a replacement for ``$\NTP_2$''
    but we prefered to use it for a (possibly) smaller class of theories. %
  } if we cannot find indiscernible sequences $\bar{a}=\left(a_{i}\right)_{i\in\mathbb{Z}}$,
  $\bar{b}=\left(b_{j}\right)_{i\in\mathbb{Z}}$ and a formula $\varphi(x,y)$
  such that:
  \begin{itemize}
  \item $a_{0}=b_{0}$,
  \item $\bar{b}$ is indiscernible over $\left(a_{i}\right)_{i\neq0}$,
  \item $\left\{ \varphi(x,a_{i})\right\} _{i\in \mathbb{Z}}$ is consistent,
  \item $\left\{ \varphi(x,b_{j})\right\} _{j \in \mathbb{Z}}$ is inconsistent.
  \end{itemize}
\end{dfn}
\begin{rmk}\label{rmk: IndexChangeResilience}
\begin{enumerate}
  \item It follows by compactness that we get an equivalent definition replacing
  $\mathbb{Z}$ by $\mathbb{Q}$ for either of $i$ or $j$ (or both), and replacing $\mathbb{Z}$ by $\omega$ for $j$.
\item If $T$ is resilient and $A$ is a set of constants, then $T(A)$ is resilient.
\end{enumerate}
\end{rmk}
\begin{lem}
  \label{lem: Equivalents to NTP2^+} The following are equivalent:
  \begin{enumerate}
  \item $T$ is resilient.
  \item For every $p,q\in S^{\EM}$, $p\leq^{+}q$ implies $p\leq^{\div}q$.
  \item For any indiscernible sequence $\bar{a}=\left(a_{i}\right)_{i\in\mathbb{Z}}$
    and $\varphi(x,y)\in L$, if $\varphi(x,a_{0})$ divides over $\left(a_{i}\right)_{i\neq0}$,
    then $\left\{ \varphi(x,a_{i})\right\} _{i\in\mathbb{Z}}$ is inconsistent.
  \item There is no array $\left(a_{ij}\right)_{i,j\in\omega}$, $\varphi(x,y)\in L$ and $k \in \omega$
    such that $\left\{ \varphi(x,a_{i0})\right\} _{i\in\omega}$ is consistent,
    $\left\{ \varphi(x,a_{ij})\right\} _{j\in\omega}$ is $k$-inconsistent
    for each $i\in\omega$ and $\bar{a}_{i}=\left(a_{ij}\right)_{j\in\omega}$
    is indiscernible over $\left(a_{j0}\right)_{j\neq i}$ for each $i\in\omega$.
  \end{enumerate}
\end{lem}
\begin{proof}
  (i) is equivalent to (ii) Assume that $p\leq^{+}q$, i.e. there is
  $\bar{a}=\left(a_{i}\right)_{i\in\mathbb{Z}}\vDash q$ and $\bar{b}=\left(b_{i}\right)_{i\in\mathbb{Z}}\vDash p$
  such that $a_{0}=b_{0}$ and $\bar{b}$ is indiscernible over $\left(a_{i}\right)_{i\neq0}$.
  For any $\varphi\left(x,y\right)$, if $\left\{ \varphi\left(x,b_{i}\right)\right\} _{i\in\omega}$
  is inconsistent, then $\left\{ \varphi\left(x,a_{i}\right)\right\} _{i\in\omega}$
  is inconsistent by resilience, which means precisely that $p\leq^{\div}q$.
  The converse is clear.

  (i) is equivalent to (iii) If $\varphi\left(x,a_{0}\right)$ divides
  over $a_{\neq0}$, then there is a sequence $\left(b_{i}\right)_{i\in\mathbb{Z}}$
  indiscernible over $a_{\neq0}$ and such that $b_{0}=a_{0}$ and $\left\{ \varphi\left(x,b_{i}\right)\right\} _{i\in\mathbb{Z}}$
  is inconsistent. It follows by resilience that $\left\{ \varphi\left(x,a_{i}\right)\right\} _{i\in\mathbb{Z}}$
  is inconsistent. On the other hand, assume that $\left\{ \varphi\left(x,a_{i}\right)\right\} _{i\in\mathbb{Z}}$
  is inconsistent. By compactness we can extend our indiscernible sequence
  to $\bar{a}'+\bar{a}+\bar{a}''=\left(a_{i}'\right)_{i\in\omega^{*}}+\left(a_{i}\right)_{i\in\mathbb{Z}}+\left(a_{i}''\right)_{i\in\omega}$.
  But then $\bar{a}$ witnesses that $\varphi\left(x,a_{0}\right)$
  divides over $\bar{a}'\bar{a}''$. Sending $\bar{a}'$ to $a_{\leq-1}$
  and $\bar{a}''$ to $a_{\geq1}$ by an automorphism fixing $a_{0}$
  we conclude that $\varphi\left(x,a_{0}\right)$ divides over $a_{\neq0}$.

  (i) is equivalent to (iv) Let $\bar{a}$, $\bar{b}$ and $\varphi\left(x,y\right)$
  witness that $T$ is not resilient. Then we let $\bar{a}_{0}=\bar{b}$
  and we let $\bar{a}_{i}$ be an image of $\bar{b}$ under some automorphism
  sending $(\ldots, a_{-1}, a_{0}, a_1, \ldots)$ to $(\ldots, a_{i-1}, a_{i}, a_{i+1}, \ldots)$ by indiscernibility. It follows that
  $\left(a_{ij}\right)_{i,j\in\omega}$ is an array as wanted.

  Conversely, if we have an array as in (iv), by compactness we may assume
  that it is of the form $\left(a_{ij}\right)_{i \in \mathbb{Z},j\in \omega}$
  and that in addition $\left(a_{i0}\right)_{i\in\mathbb{Z}}$ is indiscernible.
  Then $\bar{a}=\left(a_{i0}\right)_{i\in\mathbb{Z}}$, $\bar{b}=\left(a_{0j}\right)_{j\in\omega}$
  and $\varphi\left(x,y\right)$ contradict resilience (in view of \autoref{rmk: IndexChangeResilience}).

\end{proof}
\begin{prp}
  \label{prop: NTP2 vs NTP2^+}
  \begin{enumerate}
  \item If $T$ is $\NIP$, then it is resilient.
  \item If $T$ is simple, then it is resilient.
  \item If $T$ is resilient, then it is $\NTP_2$.
  \end{enumerate}
\end{prp}
\begin{proof}
\begin{enumerate}
  \item Fix $\varphi(x,y)$ and assume that $\left\{ \varphi(x,a_{i})\right\} _{i\in\mathbb{Q}}$
  is consistent. Then by $\NIP$ there is a maximal $k\in\omega$ such
  that $\left\{ \neg\varphi(x,a_{i})\right\} _{i\in s}\cup\left\{ \varphi(x,a_{i})\right\} _{i\notin s}$
  is consistent, for $s=\left\{ 1,2,...,k\right\} \subseteq\mathbb{Q}$.
  Let $d$ realize it. If $\left\{ \varphi(x,b_{i})\right\}_{i \in \mathbb{Q}} $ was inconsistent,
  then we would have $\neg\varphi(d,b_{i})$ for some $i\in \mathbb{Q}$,
  and thus $\left\{ \neg\varphi(x,a_{i})\right\} _{i\in s\cup\{k+1\}}\cup\left\{ \varphi(x,a_{i})\right\} _{i\notin s\cup\{k+1\}}$
  would be consistent, by all the indiscernibility around --- a
  contradiction to the maximality of $k$. Thus, $\left\{ \varphi(x,b_{i})\right\} _{i\in\mathbb{Q}}$
  is consistent.

  \item It is easy to see that $\left(a_{i}\right)_{i>0}$ is a Morley
  sequence over $A=\left(a_{i}\right)_{i<0}$ by finite satisfiability.
  If $\varphi(x,a_{0})$ divides over $a_{\neq0}$, then by Kim's lemma
  $\left\{ \varphi(x,a_{i})\right\} _{i\in\mathbb{Z}}$ is inconsistent.

  \item By Erd\H os-Rado and compactness we can find a strongly indiscernible
  array $\left(c_{ij}\right)_{i,j\in\mathbb{Z}}$ witnessing $\TP_{2}$
  for $\varphi\left(x,y\right)$. Set $a_{i}=c_{i0}$ for $i\in\omega$
  and $b_{j}=b_{0j}$ for $j\in\omega$. Then $\bar{a}$, $\bar{b}$
  and $\varphi\left(x,y\right)$ witness that $T$ is not resilient.
\end{enumerate}
\end{proof}
\begin{clm}
  \label{clm: universal if independent enough} Let $T$ be resilient,
  $A$ an extension base, and let $\bar{a}=(a_{i})_{i\in\mathbb{Z}}$
  be indiscernible over $A$, say in and $r=\tp(a_{0}/A)\in S(A)$.
  Then the following are equivalent:
  \begin{enumerate}
  \item The EM type $\tp^{\EM}(\bar{a}/A)$ is $\leq_{A}^{\div}$-greatest
    in $S^{\EM,r}(A)$.
  \item $\tp(a_{\neq0}/a_{0}A)$ does not divide over $A$.
  \end{enumerate}
\end{clm}
\begin{proof}
  We may assume that $A=\emptyset$.

  (i) implies (ii) in any theory: Let $\vDash\varphi(a_{\neq0},a_{0})$.
  By indiscernibility and compactness $\left\{ \varphi(x,a_{i})\right\} _{i\in\mathbb{Z}}$
  is consistent, so by (i) $\varphi(x,a_{0})$ does not divide.

  (ii) implies (i): Assume that $\varphi(x,a_{0})$ divides. As $\tp(a_{\neq0}/a_{0})$
  does not divide, it follows that $\varphi(x,a_{0})$ divides over
  $a_{\neq0}$. But then by \autoref{lem: Equivalents to NTP2^+}(iii)
  we have that $\left\{ \varphi(x,a_{i})\right\} _{i\in\mathbb{Z}}$
  is inconsistent, hence (i).
\end{proof}
\begin{rmk}
  Similar observation in the context of NIP theories based on \cite{Shelah:DependentContinued}
  is made in \cite{Kaplan-Usvyatsov:StrictIndependence}.
\end{rmk}
Recall that a theory is called \emph{low} if for every formula $\varphi\left(x,y\right)$
there is $k\in\omega$ such that for any indiscernible sequence $\left(a_{i}\right)_{i\in\omega}$,
$\left\{ \varphi\left(x,a_{i}\right)\right\} _{i\in\omega}$ is consistent
if and only if it is $k$-consistent. The following is a generalization
of \cite[Lemma 2.3]{BenYaacov-Pillay-Vassiliev:LovelyPairs}.
\begin{prp}
  \label{prop: lowness} Let $T$ be resilient. Then the following are
  equivalent:
  \begin{enumerate}
  \item $\varphi(x,y)$ is low.
  \item The set $\left\{ \left(c,d\right):\varphi(x,c)\mbox{ divides over }d\right\} $
    is type-definable (where $d$ is allowed to be of infinite length).
  \end{enumerate}
\end{prp}
\begin{proof}
  (i) implies (ii) holds in any theory, and we show that (ii) implies
  (i).

  Assume that $\varphi\left(x,y\right)$ is not low. Then for every
  $i\in\omega$ we have a sequence $\bar{a}_{i}=\left(a_{ij}\right)_{j\in\mathbb{Z}}$
  such that $\left\{ \varphi\left(x,a_{ij}\right)\right\} _{j\in\mathbb{Z}}$
  is $i$-consistent, but inconsistent. In particular $\varphi\left(x,a_{i0}\right)$
  divides over $\left(a_{ij}\right)_{j\neq0}$ for each $i$.

  If (ii) holds, then by compactness we can find a sequence $\bar{a}=\left(a_{j}\right)_{j\in\omega}$
  such that $\left\{ \varphi\left(x,a_{j}\right)\right\} _{j\in\omega}$
  is consistent and $\varphi\left(x,a_{0}\right)$ still divides over $a_{\neq0}$.
  But this is a contradiction to resilience by \autoref{lem: Equivalents to NTP2^+}(iii).
\end{proof}
However, the main question remains unresolved:
\begin{qst}
  \begin{enumerate}
  \item Does $\NTP_2$ imply resilience?
  \item Is resilience preserved under reducts?
  \item Does type-definability of dividing imply lowness in $\NTP_2$ theories?
  \end{enumerate}
\end{qst}

\section{On a strengthening of strong theories}

Recently several attempts have been made to define weight outside
of the familiar context of simple theories. First Shelah had defined
strongly dependent theories and several notions of dp-rank in \cite{Shelah:DependentContinued,Shelah:StronglyDependent}.
The study of dp-rank was continued in \cite{Onshuus-Usvyatsov:dpMinimality}. After
that Adler \cite{Adler:Burden} had introduced \emph{burden}, a notion
based on the invariant $\kappa_{\inp}$ of Shelah \cite{Shelah:ClassificationTheory}
which generalizes simultaneously dp-rank in NIP theories and weight
in simple theories. In this section we are going to add yet another
version of measuring weight. First we recall the notions mentioned
above.

For notational convenience we consider an extension $\card$ of the
linear order on cardinals by adding a new maximal element $\infty$
and replacing every limit cardinal $\kappa$ by two new elements $\kappa_{-}$
and $\kappa_{+}$. The standard embedding of cardinals into $\card$
identifies $\kappa$ with $\kappa_{+}$. In the following, whenever
we take a supremum of a set of cardinals, we will be computing it
in $\card$.
\begin{dfn}
  \cite{Adler:Burden} Let $p\left(x\right)$ be a (partial) type.
  \begin{enumerate}
  \item An $\inp$-pattern of depth $\kappa$ in $p(x)$ consists of $\left(\bar{a}_{i},\varphi_{i}(x,y_{i}),k_{i}\right)_{i\in\kappa}$
    with $\bar{a}_{i}=\left(a_{ij}\right)_{j\in\omega}$ and $k_{i}\in\omega$
    such that:

    \begin{itemize}
    \item $\left\{ \varphi_{i}(x,a_{ij})\right\} _{j\in\omega}$ is $k_{i}$-inconsistent
      for every $i\in\kappa$,
    \item $p(x)\cup\left\{ \varphi_{i}(x,a_{if(i)})\right\} _{i\in\kappa}$
      is consistent for every $f:\,\kappa\to\omega$.
    \end{itemize}
  \item The \emph{burden} of a partial type $p(x)$ is the supremum (in Card$^{*}$)
    of the depths of $\inp$-patterns in it. We denote the burden of $p$
    as $\bdn(p)$ and we write $\bdn(a/A)$ for $\bdn(\tp(a/A))$.
  \item We get an equivalent definition by taking supremum only over $\inp$-patterns
    with mutually indiscernible rows.
  \item It is easy to see by compactness that $T$ is $\NTP_2$ if and only
    if $\bdn\left(\mbox{"}x=x\mbox{"}\right)<\infty$, if and only if
    $\bdn\left(\mbox{"}x=x\mbox{"}\right)<\left|T\right|^{+}$.
  \item A theory $T$ is called \emph{strong} if $\bdn\left(p\right)\leq\left(\aleph_{0}\right)_{-}$
    for every finitary type $p$ (equivalently, there is no $\inp$-pattern
    of infinite depth). Of course, if $T$ is strong then it is $\NTP_2$.
  \end{enumerate}
\end{dfn}
\begin{fct}
  \label{fct: burden in simple and NIP}\cite{Adler:Burden}
  \begin{enumerate}
  \item Let $T$ be $\NIP$. Then $\bdn(p)=\dprk(p)$ for any $p$.
  \item Let $T$ be simple. Then the burden of $p$ is the supremum of weights
    of its complete extensions.
  \end{enumerate}
\end{fct}

Some basics of the theory of burden were developed by the second author
in \cite{Chernikov:NTP2}.
\begin{fct}
  \label{fac: burden is sub-multiplicative}\cite{Chernikov:NTP2} Let $T$
  be an arbitrary theory.
  \begin{enumerate}
  \item The following are equivalent:

    \begin{enumerate}
    \item $\bdn(p)<\kappa$.
    \item For any $\left(\bar{a}_{i}\right)_{i\in\kappa}$ mutually indiscernible
      over $A$ and $b\vDash p$, there is some $i\in\kappa$ and $\bar{a}_{i}'$
      such that $\bar{a}_{i}'$ is indiscernible over $bA$ and $\bar{a}_{i}'\equiv_{Aa_{i0}}\bar{a}_{i}$.
    \end{enumerate}
  \item Assume that $\bdn(a/A)<\kappa$ and $\bdn(b/aA)<\lambda$, with $\kappa$
    and $\lambda$ finite or infinite cardinals. Then $\bdn(ab/A)<\kappa\times\lambda$.
  \item In particular, in the definition of strong (or $\NTP_2$) it is
    enough to look at types in one variable.
  \end{enumerate}
\end{fct}
In \cite{Kaplan-Onshuus-Usvyatsov:dpRank} it is proved that dp-rank is sub-additive,
so burden in NIP theories is sub-additive as well. The sub-additivity
of burden in simple theories follows from \autoref{fct: burden in simple and NIP}
and the sub-additivity of weight in simple theories. It thus becomes
natural to wonder if burden is sub-additive in general, or at least
in $\NTP_2$ theories.

~

Now we are going to define a refinement of the class of strong theories.
\begin{dfn}
  Let $p\left(x\right)$ be a partial type.
  \begin{enumerate}
  \item An $\inp^{2}$-pattern of depth $\kappa$ in $p\left(x\right)$ consists
    of formulas $\left(\varphi_{i}(x,y_{i},z_{i})\right)_{i\in\kappa}$,
    mutually indiscernible sequences $\left(\bar{a}_{i}\right)_{i\in\kappa}$
    and $b_{i}\subseteq\bigcup_{j<i}\bar{a}_{j}$ such that:

    \begin{enumerate}
    \item $\left\{ \varphi_{i}(x,a_{i0},b_{i})\right\} _{i\in\omega}\cup p\left(x\right)$
      is consistent,
    \item $\left\{ \varphi_{i}(x,a_{ij},b_{i})\right\} _{j\in\omega}$ is inconsistent
      for every $i\in\omega$.
    \end{enumerate}
  \item An $\inp^{3}$-pattern of depth $\kappa$ in $p\left(x\right)$ is
    defined exactly as an $\inp^{2}$-pattern of depth $\kappa$, but
    allowing $b_{i}\subseteq\bigcup_{j\in\kappa,j\neq i}\bar{a}_{j}$.
    It is then clear that every $\inp^{2}$-pattern is an $\inp^{3}$-pattern
    of the same depth, but the opposite is not true.
  \item The \emph{burden$^{2}$} (\emph{burden}$^{3}$) of a partial type
    $p(x)$ is the supremum (in Card$^{*}$) of the depths of $\inp^{2}$-patterns
    (resp. $\inp^{3}$-patterns) in it. We denote the burden$^{2}$ of
    $p$ as $\bdn^{2}(p)$ and we write $\bdn^{2}(a/A)$ for $\bdn^{2}(\tp(a/A))$
    (and similarly for $\bdn^{3}$).
  \item A theory $T$ is called \emph{strong$^{2}$} if $\bdn^{2}\left(p\right)\leq\left(\aleph_{0}\right)_{-}$
    for every finitary type $p$ (that is, there is no $\inp^{2}$-pattern
    of infinite depth). Similarly for \emph{strong}$^{3}$.
  \end{enumerate}
\end{dfn}

In the following proposition we sum up some of the properties of $\bdn^{2}$
and $\bdn^{3}$.
\begin{prp}
  \label{prop: burden^2 summary}
  \begin{enumerate}
  \item \label{enu: bdn inequalities} For any partial type $p\left(x\right)$,
    $\bdn\left(p\right)\leq\bdn^{2}\left(p\right)\leq\bdn^{3}\left(p\right)$.
  \item \label{enu: Strong-implies-strong2} Strong$^{3}$ implies strong$^{2}$
    implies strong.
  \item \label{enu: strong2 =00003D strong3} In fact, $T$ is strong$^{2}$
    if and only if it is strong$^{3}$.
  \item \label{enu: strongly2 dependent}$T$ is strongly$^{2}$ dependent
    if and only if it is NIP and strong$^{2}$ (we recall from \cite[Definition~2.2]{Kaplan-Shelah:ChainConditionsInDependentGroups}
    that $T$ is called strongly$^{2}$ dependent when there are no $\left(\varphi_{i}\left(x,y_{i},z_{i}\right),\,\bar{a}_{i}=\left(a_{ij}\right)_{j\in\omega},\, b_{i}\subseteq\bigcup_{j<i}\bar{a}_{j}\right)_{i\in\omega}$
    such that $\left(\bar{a}_{i}\right)_{i\in\omega}$ are mutually indiscernible
    and the set $\left\{ \varphi_{i}\left(x,a_{i0},b_{i}\right)\land\neg\varphi_{i}\left(x,a_{i1},b_{i}\right)\right\} _{i\in\omega}$
    is consistent.).
  \item \label{enu: supersimple} If $T$ is supersimple, then it is strong$^{2}$.
  \item \label{enu: strong2 not supersimple} There are strong$^{2}$ stable
    theories which are not superstable.
  \item \label{enu: strong, not strong2} There are strong stable theories
    which are not strong$^{2}$.
  \item \label{enu: NTP2 iff bdd burden3} We still have that $T$ is $\NTP_2$
    if and only if every finitary type has bounded burden$^{3}$.
  \end{enumerate}
\end{prp}
\begin{proof}
  \ref{enu: bdn inequalities} is immediate by comparing the definitions,
  and \ref{enu: Strong-implies-strong2} follows from \ref{enu: bdn inequalities}.

 \ref{enu: strong2 =00003D strong3} Assume that $T$ is not strong$^{3}$,
  witnessed by $\left(\varphi_{i}(x,y_{i},z_{i}),\bar{a}_{i},b_{i}\right)_{i\in\omega}$.
  For $i\in\omega$, let $f\left(i\right)$ be the smallest $j\in\omega$
  such that $b_{i}\in\bar{a}_{<j}$. Now for $i\in\omega$ we define
  inductively:
  \begin{itemize}
  \item $\alpha_{0}=0$, $\alpha_{i+1}=f\left(\alpha_{i}\right)$,
  \item $b_{i}'=b_{\alpha_{i}}\cap\bar{a}_{\in\left\{ \alpha_{0},\alpha_{1},\ldots,\alpha_{i-1}\right\} }$
    and $b_{i}''=b_{\alpha_{i}}\cap\bar{a}_{\in\left\{ 0,1,\ldots,\alpha_{i+1}-1\right\} \setminus\left\{ \alpha_{0},\alpha_{1},\ldots,\alpha_{i}\right\} }$,
    so we may assume that $b_{\alpha_{i}}=b_{i}''b_{i}'$.
  \item $a_{ij}'=a_{\alpha_{i}j}b_{i}''$ for $j\in\omega$,
  \item $\varphi'_{i}\left(x,a_{ij}',b_{i}'\right)=\varphi_{i}\left(x,a_{ij},b_{i}\right)$.
  \end{itemize}
  It is now easy to check that $\left(\bar{a}_{i}'\right)_{i\in\omega}$
  are mutually indiscernible, $b_{i}'\in\bar{a}_{<i}'$, $\left\{ \varphi'_{i}\left(x,a_{i0}',b_{i}'\right)\right\} _{i\in\omega}$
  is consistent and $\left\{ \varphi'_{i}\left(x,a_{ij}',b_{i}'\right)\right\} _{j\in\omega}$
  is inconsistent for every $i\in\omega$. This gives us an $\inp^{2}$-pattern
  of infinite depth, witnessing that $T$ is not strong$^{2}$.

  \ref{enu: strongly2 dependent} Let $\left(\varphi_{i}\left(x,y_{i},z_{i}\right),\bar{a}_{i},b_{i}\right)_{i\in\omega}$
  witness that $T$ is not strong$^{2}$ and let $c\vDash\left\{ \varphi_{i}(x,a_{i0},b_{i})\right\} _{i\in\omega}$,
  it follows from the inconsistency of $\left\{ \varphi\left(x,a_{ij},b_{i}\right)\right\} _{j\in\omega}$'s
  that for each $i\in\omega$ there is some $k_{i}\in\omega$ such that
  $c\vDash\left\{ \varphi_{i}(x,a_{i0},b_{i})\land\neg\varphi_{i}\left(x,a_{ik_{i}},b_{i}\right)\right\} _{i\in\omega}$.
  Define $a_{ij}'=a_{i,k_{i}\times j}a_{i,k_{i}\times j+1}\ldots a_{i,k_{i}\times\left(j+1\right)-1}$
  and $\varphi'\left(x,a_{ij}',b_{i}\right)=\varphi\left(x,a_{i,k_{i}\times j},b_{i}\right)$.
  Then $\left(\bar{a}_{i}'\right)_{i\in\omega}$ are mutually indiscernible,
  $b_{i}\in\bigcup_{j<i}\bar{a}_{j}'$ and $c\vDash\left\{ \varphi_{i}\left(x,a_{i0}',b_{i}\right)\land\neg\varphi_{i}\left(x,a_{i1}',b_{i}\right)\right\} _{i\in\omega}$
  --- witnessing that $T$ is not strongly$^{2}$ dependent.

  On the other hand, let $\left(\varphi_{i}\left(x,y_{i},z_{i}\right),\bar{a}_{i},b_{i}\right)_{i\in\omega}$
  witness that $T$ is not strongly$^{2}$ dependent and assume that
  $T$ is NIP. Let $\varphi'_{i}\left(x,y_{i}',z_{i}\right)=\varphi_{i}\left(x,y_{i}^{0},z_{i}\right)\land\neg\varphi_{i}\left(x,y_{i}^{1},z_{i}\right)$,
  $a_{ij}'=a_{i\left(2j\right)}a_{i\left(2j+1\right)}$ for all $i,j\in\omega$.
  We then have that $\left(\bar{a}_{i}'\right)_{i\in\omega}$ are still
  mutually indiscernible and $b_{i}\in\bigcup_{j<i}\bar{a}'$, $\left\{ \varphi'_{i}\left(x,a_{i0}',b_{i}\right)\right\} _{i\in\omega}$
  is consistent and $\left\{ \varphi_{i}'\left(x,a_{ij}',b_{i}\right)\right\} _{j\in\omega}$
  is inconsistent (otherwise let $c$ realize it, it follows that $\varphi_{i}\left(c,a_{ij},b_{i}\right)$
  holds if and only if $j$ is even, contradicting NIP). But this shows
  that $T$ is not strong$^{2}$.

  \ref{enu: supersimple} Let $T$ be supersimple, and assume that
  $T$ is not strong$^{2}$, witnessed by $\left(\varphi_{i}\left(x,y_{i},z_{i}\right),\bar{a}_{i},b_{i}\right)_{i\in\omega}$
  and let $A=\bigcup_{i,j\in\omega}a_{ij}$. Let $c\vDash\left\{ \varphi_{i}(x,a_{i0},b_{i})\right\} _{i\in\omega}$.
  By supersimplicity, there has to be some finite $A_{0}\subset A$
  such that $\tp\left(c/A\right)$ does not divide over $A_{0}$. It
  follows that there is some $i'\in\omega$ such that $A_{0}\subset\bigcup_{i<i',j\in\omega}a_{ij}$.
  But then $c\vDash\varphi_{i'}\left(x,a_{i'0},b_{i'}\right)$, $\left(a_{i'j}b_{i'}\right)_{j\in\omega}$
  is indiscernible over $A_{0}$ and $\left\{ \varphi\left(x,a_{i'j},b_{i'}\right)\right\} _{j\in\omega}$
  is inconsistent, so $\tp\left(c/A\right)$ divides over $A_{0}$ ---
  a contradiction.

  \ref{enu: strong2 not supersimple} It is easy to see that the theory
  of an infinite family of refining equivalence relations with infinitely
  many infinite classes satisfies the requirement.

  \ref{enu: strong, not strong2} In \cite[Example~2.5]{Shelah:StronglyDependent} Shelah
  gives an example of a strongly stable theory which is not strongly$^{2}$
  stable. In view of (3) this is sufficient. Besides, there are examples
  of NIP theories of burden 1 which are not strongly$^{2}$ dependent
  (e.g. $\left(\mathbb{Q}_{p},+,\cdot,0,1\right)$ or $\left(\mathbb{R},<,+,\cdot,0,1\right)$).

  \ref{enu: NTP2 iff bdd burden3} We remind the statement of Fodor's
  lemma.

  \textbf{Fact} (Fodor's lemma). If $\kappa$ is a regular, uncountable
  cardinal and $f:\,\kappa\to\kappa$ is such that $f(\alpha)<\alpha$
  for any $\alpha\neq0$, then there is some $\gamma$ and some stationary
  $S\subseteq\kappa$ such that $f(\alpha)=\gamma$ for any $\alpha\in S$.

  If $T$ has $\TP_{2}$, then clearly $\bdn^{3}\left(T\right)=\infty$,
  and we prove the converse. Assume that $\bdn^{3}\left(T\right)\geq\left|T\right|^{+}$
  and let $\kappa=\left|T\right|^{+}$. Then we can find $\left(\varphi_{i}\left(x,y_{i},z_{i}\right),\bar{a}_{i},b_{i}\right)_{i\in\kappa}$
  with $\left(\bar{a}_{i}\right)_{i\in\kappa}$ mutually indiscernible,
  finite $b_{i}\in\bigcup_{j\in\kappa,j\neq i}\bar{a}_{j}$ such that
  $\left\{ \varphi_{i}(x,a_{i0},b_{i})\right\} _{i\in\kappa}$ is consistent
  and $\left\{ \varphi_{i}(x,a_{ij},b_{i})\right\} _{j\in\omega}$ is
  inconsistent for every $i\in\kappa$. For each $i\in\kappa$, let
  $f\left(i\right)$ be the largest $j<i$ such that $\bar{a}_{j}\cap b_{i}\neq\emptyset$
  and let $g\left(i\right)$ be the largest $j\in\kappa$ such that
  $\bar{a}_{j}\cap b_{i}\neq\emptyset$. By Fodor's lemma there is some
  stationary $S\subseteq\kappa$ and $\gamma\in\kappa$ such that $f(i)=\gamma$
  for all $i\in S$.

  By induction we choose an increasing sequence $\left(i_{\alpha}\right)_{\alpha\in\kappa}$
  from $S$ such that $i_{0}>\gamma$ and $i_{\alpha}>g(i_{\beta})$
  for $\beta<\alpha$. Now let $a_{\alpha j}'=a_{i_{\alpha}j}b_{i_{\alpha}}$
  and $\varphi_{\alpha}'\left(x,y_{\alpha}'\right)=\varphi_{i_{\alpha}}\left(x,y_{i_{\alpha}},z_{i_{\alpha}}\right)$.
  It follows by the choice of $i_{\alpha}$'s that $\left(\bar{a}_{\alpha}'\right)_{\alpha\in\kappa}$
  are mutually indiscernible, $\left\{ \varphi_{\alpha}'\left(x,a_{\alpha0}'\right)\right\} _{\alpha\in\kappa}$
  is consistent and $\left\{ \varphi_{\alpha}'\left(x,a_{\alpha j}'\right)\right\} _{j\in\omega}$
  is inconsistent for each $\alpha\in\kappa$. It follows that we had
  found an $\inp$-pattern of depth $\kappa=\left|T\right|^{+}$ ---
  so $T$ has $\TP_{2}$.
\end{proof}
We are going to give an analogue of Fact \ref{fac: burden is sub-multiplicative}(1)
for burden$^{2,3}$, but first a standard lemma.
\begin{lem}
  \label{lem: make indiscernible if consistent} Let $\bar{a}=\left(a_{i}\right)_{i\in\omega}$
  be indiscernible over $A$ and let $p(x,a_{0})=\tp(c/a_{0}A)$. Assume
  that $\left\{ p(x,a_{i})\right\} _{i\in\omega}$ is consistent. Then
  there is $\bar{a}'\equiv_{a_{0}A}\bar{a}$ which is indiscernible
  over $cA$.
\end{lem}

\begin{lem}
  Let $p\left(x\right)$ be a partial type over $A$:
  \begin{enumerate}
  \item The following are equivalent:

    \begin{enumerate}
    \item $\bdn^{3}\left(p\right)<\kappa$.
    \item For any $\left(\bar{a}_{i}\right)_{i\in\kappa}$ mutually indiscernible
      over $A$ and $c\vDash p\left(x\right)$ there is some $i\in\kappa$
      and $\bar{a}_{i}'$ such that:

      \begin{itemize}
      \item $\bar{a}_{i}'\equiv_{a_{i0}\bar{a}_{\neq i}A}\bar{a}_{i}$,
      \item $\bar{a}_{i}'$ is indiscernible over $c\bar{a}_{\neq i}A$.
      \end{itemize}
    \end{enumerate}
  \item The following are equivalent:

    \begin{enumerate}
    \item $\bdn^{2}\left(p\right)<\kappa$.
    \item For any $\left(\bar{a}_{i}\right)_{i\in\kappa}$ mutually indiscernible
      over $A$ and $c\vDash p\left(x\right)$ there is some $i\in\kappa$
      and $\bar{a}_{i}'$ such that:

      \begin{itemize}
      \item $\bar{a}_{i}'\equiv_{a_{i0}\bar{a}_{<i}A}\bar{a}_{i}$,
      \item $\bar{a}_{i}'$ is indiscernible over $c\bar{a}_{<i}A$.
      \end{itemize}
    \end{enumerate}
  \end{enumerate}
\end{lem}
\begin{proof}
  (i): (a) implies (b): Let $\left(\bar{a}_{i}\right)_{i\in\kappa}$
  mutually indiscernible over $A$ and $c\vDash p\left(x\right)$ be
  given. Define $p_{i}\left(x,a_{i0}\right)=\tp\left(c/a_{i0}\bar{a}_{\neq i}A\right)$.
  By \autoref{lem: make indiscernible if consistent} it is enough
  to show that $\bigcup_{j\in\omega}p_{i}\left(x,a_{ij}\right)$ is
  consistent for some $i\in\kappa$.

  Assume not, but then by compactness for each $i\in\kappa$ we have
  some $\varphi_{i}\left(x,a_{i0},b_{i}d_{i}\right)\in p_{i}\left(x,a_{i0}\right)$
  with $b_{i}\in\bar{a}_{\neq i}$ and $d_{i}\in A$ such that $\left\{ \varphi_{i}\left(x,a_{ij},b_{i}d_{i}\right)\right\} _{j\in\omega}$
  is inconsistent. Let $\varphi'_{i}\left(x,a_{ij}',b'_{i}\right)=\varphi_{i}\left(x,a_{ij},b_{i}d_{i}\right)$
  with $a_{ij}'=a_{ij}d_{i}$ and $b_{i}'=b_{i}$. It follows that $\left(\bar{a}_{i}'\right)_{i\in\kappa}$
  are mutually indiscernible, $c\vDash\left\{ \varphi'_{i}\left(x,a_{i0}',b_{i}'\right)\right\} _{i\in\kappa}\cup p\left(x\right)$
  and $\left\{ \varphi'_{i}\left(x,a_{ij}',b_{i}'\right)\right\} _{j\in\omega}$
  is inconsistent for each $i\in\kappa$, thus witnessing that $\bdn^{3}\left(p\right)\geq\kappa$
  --- a contradiction.

  (b) implies (a): Assume that $\bdn^{3}\left(p\right)\geq\kappa$,
  witnessed by an $\inp^{3}$-pattern $\left(\varphi_{i}\left(x,y_{i},z_{i}\right),\bar{a}_{i},b_{i}\right)_{i\in\kappa}$
  in $p\left(x\right)$. Let $c\vDash\left\{ \varphi_{i}\left(x,a_{i0},b_{i}\right)\right\} _{i\in\kappa}$
  and take $A=\emptyset$. It is then easy to check that (2) fails.

  (ii): Similar.
\end{proof}

\providecommand{\bysame}{\leavevmode\hbox to3em{\hrulefill}\thinspace}


\begin{thebibliography}{{Ben}03}

\bibitem[Adla]{Adler:IntroductionToDependent}
Hans \bgroup\scshape{}Adler\egroup{},
  \href{http://www.logic.univie.ac.at/~adler/docs/nip.pdf} {\emph{An
  introduction to theories without the independence property}}, Archive for
  Mathematical Logic, to appear.

\bibitem[Adlb]{Adler:Chain}
\bysame, \emph{Pre-independence relations}, preprint.

\bibitem[Adlc]{Adler:Burden}
\bysame, \href{http://www.logic.univie.ac.at/~adler/docs/strong.pdf}
  {\emph{Strong theories, burden, and weight}}, preprint.

\bibitem[Adl09]{Adler:ThornForking}
\bysame, \emph{Thorn-forking as local forking}, Journal of Mathematical Logic
  \textbf{9} (2009), no.~1, 21--38,
  \href{http://dx.doi.org/10.1142/S0219061309000823}{doi:10.1142/S0219061309000823}.

\bibitem[{Ben}03]{BenYaacov:SimplicityInCats}
Itaï \bgroup\scshape{}{Ben Yaacov}\egroup{},
  \href{http://math.univ-lyon1.fr/~begnac/articles/catsim.pdf}
  {\emph{Simplicity in compact abstract theories}}, Journal of Mathematical
  Logic \textbf{3} (2003), no.~2, 163--191,
  \href{http://dx.doi.org/10.1142/S0219061303000297}{doi:10.1142/S0219061303000297}.

\bibitem[BPV03]{BenYaacov-Pillay-Vassiliev:LovelyPairs}
Itaï \bgroup\scshape{}{Ben Yaacov}\egroup{}, Anand
  \bgroup\scshape{}Pillay\egroup{}, and Evgueni
  \bgroup\scshape{}Vassiliev\egroup{},
  \href{http://math.univ-lyon1.fr/~begnac/articles/pairs.pdf} {\emph{Lovely
  pairs of models}}, Annals of Pure and Applied Logic \textbf{122} (2003),
  no.~1-3, 235--261,
  \href{http://dx.doi.org/10.1016/S0168-0072(03)00018-6}{doi:10.1016/S0168-0072(03)00018-6}.

\bibitem[Cas03]{Casanovas:Dividing}
Enrique \bgroup\scshape{}Casanovas\egroup{}, \emph{Dividing and chain
  conditions}, Archive for Mathematical Logic \textbf{42} (2003), no.~8,
  815--819,
  \href{http://dx.doi.org/10.1007/s00153-003-0192-0}{doi:10.1007/s00153-003-0192-0}.

\bibitem[Che]{Chernikov:NTP2}
Artem \bgroup\scshape{}Chernikov\egroup{}, \emph{Theories without the tree
  property of the second kind}, reprint,
  \href{http://arxiv.org/abs/1204.0832}{arXiv:1204.0832}.

\bibitem[CK12]{Chernikov-Kaplan:ForkingDividingNTP2}
Artem \bgroup\scshape{}Chernikov\egroup{} and Itay
  \bgroup\scshape{}Kaplan\egroup{}, \emph{Forking and dividing in {NTP}$_2$
  theories}, Journal of Symbolic Logic \textbf{77} (2012), no.~1, 1--20,
  \href{http://arxiv.org/abs/0906.2806}{arXiv:0906.2806}.

\bibitem[CKS12]{Chernikov-Kaplan-Shelah:NonForkingSpectra}
Artem \bgroup\scshape{}Chernikov\egroup{}, Itay
  \bgroup\scshape{}Kaplan\egroup{}, and Saharon
  \bgroup\scshape{}Shelah\egroup{}, \emph{On non-forking spectra}, preprint,
  2012, \href{http://arxiv.org/abs/1205.3101}{arXiv:1205.3101}.

\bibitem[CLPZ01]{Casanovas-Lascar-Pillay-Ziegler:GaloisGroups}
Enrique \bgroup\scshape{}Casanovas\egroup{}, Daniel
  \bgroup\scshape{}Lascar\egroup{}, Anand \bgroup\scshape{}Pillay\egroup{}, and
  Martin \bgroup\scshape{}Ziegler\egroup{}, \emph{Galois groups of first order
  theories}, Journal of Mathematical Logic \textbf{1} (2001), no.~2, 305--319,
  \href{http://dx.doi.org/10.1142/S0219061301000119}{doi:10.1142/S0219061301000119}.

\bibitem[Dol04]{Dolich:WeakDividing}
Alfred \bgroup\scshape{}Dolich\egroup{}, \emph{Weak dividing, chain conditions,
  and simplicity}, Archive for Mathematical Logic \textbf{43} (2004), no.~2,
  265--283,
  \href{http://dx.doi.org/10.1007/s00153-003-0176-0}{doi:10.1007/s00153-003-0176-0}.

\bibitem[GIL02]{Grossberg-Iovino-Lessman:SimplePrimer}
Rami \bgroup\scshape{}Grossberg\egroup{}, Jos{\'e}
  \bgroup\scshape{}Iovino\egroup{}, and Olivier
  \bgroup\scshape{}Lessmann\egroup{}, \emph{A primer of simple theories},
  Archive for Mathematical Logic \textbf{41} (2002), no.~6, 541--580,
  \href{http://dx.doi.org/10.1007/s001530100126}{doi:10.1007/s001530100126}.

\bibitem[HP11]{Hrushovski-Pillay:NIPInvariantMeasures}
Ehud \bgroup\scshape{}Hrushovski\egroup{} and Anand
  \bgroup\scshape{}Pillay\egroup{}, \emph{On {NIP} and invariant measures},
  Journal of the European Mathematical Society (JEMS) \textbf{13} (2011),
  no.~4, 1005--1061,
  \href{http://dx.doi.org/10.4171/JEMS/274}{doi:10.4171/JEMS/274}.

\bibitem[Hru12]{Hrushovski:ApproximateSubgroups}
Ehud \bgroup\scshape{}Hrushovski\egroup{}, \emph{Stable group theory and
  approximate subgroups}, Journal of the American Mathematical Society
  \textbf{25} (2012), no.~1, 189--243,
  \href{http://dx.doi.org/10.1090/S0894-0347-2011-00708-X}{doi:10.1090/S0894-0347-2011-00708-X}.

\bibitem[HZ96]{Hrushovski-Zilber:ZariskiGeometries}
Ehud \bgroup\scshape{}Hrushovski\egroup{} and Boris
  \bgroup\scshape{}Zilber\egroup{}, \emph{Zariski geometries}, Journal of the
  American Mathematical Society \textbf{9} (1996), no.~1, 1--56,
  \href{http://dx.doi.org/10.1090/S0894-0347-96-00180-4}{doi:10.1090/S0894-0347-96-00180-4}.

\bibitem[Kim96]{Kim:PhD}
Byunghan \bgroup\scshape{}Kim\egroup{},
  \href{http://gateway.proquest.com/openurl?url_ver=Z39.88-2004&rft_val_fmt=info:ofi/fmt:kev:mtx:dissertation&res_dat=xri:pqdiss&rft_dat=xri:pqdiss:9629800}
  {\emph{Simple first order theories}}, Ph.D. thesis, University of Notre Dame,
  1996, p.~96.

\bibitem[KOU]{Kaplan-Onshuus-Usvyatsov:dpRank}
Itay \bgroup\scshape{}Kaplan\egroup{}, Alf \bgroup\scshape{}Onshuus\egroup{},
  and Alexander \bgroup\scshape{}Usvyatsov\egroup{}, \emph{Additivity of the
  dp-rank}, Transactions of the American Mathematical Society, to appear,
  \href{http://arxiv.org/abs/1109.1601}{arXiv:1109.1601}.

\bibitem[KS12]{Kaplan-Shelah:ChainConditionsInDependentGroups}
Itay \bgroup\scshape{}Kaplan\egroup{} and Saharon
  \bgroup\scshape{}Shelah\egroup{}, \emph{Chain conditions in dependent
  groups}, preprint, 2012,
  \href{http://arxiv.org/abs/1112.0807}{arXiv:1112.0807}.

\bibitem[KU]{Kaplan-Usvyatsov:StrictIndependence}
Itay \bgroup\scshape{}Kaplan\egroup{} and Alexander
  \bgroup\scshape{}Usvyatsov\egroup{}, \emph{Strict independence in dependent
  theories}, In preparation.

\bibitem[Les00]{Lessmann:CountingPartialTypes}
Olivier \bgroup\scshape{}Lessmann\egroup{}, \emph{Counting partial types in
  simple theories}, Colloquium Mathematicum \textbf{83} (2000), no.~2,
  201--208.

\bibitem[OU11]{Onshuus-Usvyatsov:dpMinimality}
Alf \bgroup\scshape{}Onshuus\egroup{} and Alexander
  \bgroup\scshape{}Usvyatsov\egroup{}, \emph{On dp-minimality, strong
  dependence and weight}, Journal of Symbolic Logic \textbf{76} (2011), no.~3,
  737--758,
  \href{http://dx.doi.org/10.2178/jsl/1309952519}{doi:10.2178/jsl/1309952519}.

\bibitem[Poi85]{Poizat:Cours}
Bruno \bgroup\scshape{}Poizat\egroup{}, \emph{Cours de théorie des modèles},
  Nur al-Mantiq wal-Ma'rifah, Lyon, 1985, Une introduction à la logique
  mathématique contemporaine.

\bibitem[She]{Shelah:StronglyDependent}
Saharon \bgroup\scshape{}Shelah\egroup{}, \emph{Strongly dependent theories},
  preprint, \href{http://arxiv.org/abs/math.LO/0504197}{arXiv:math.LO/0504197}.

\bibitem[She80]{Shelah:SimpleUnstableTheories}
\bysame, \emph{Simple unstable theories}, Annals of Mathematical Logic
  \textbf{19} (1980), no.~3, 177--203,
  \href{http://dx.doi.org/10.1016/0003-4843(80)90009-1}{doi:10.1016/0003-4843(80)90009-1}.

\bibitem[She90]{Shelah:ClassificationTheory}
\bysame, \emph{Classification theory and the number of nonisomorphic models},
  second ed., Studies in Logic and the Foundations of Mathematics, vol.~92,
  North-Holland Publishing Co., Amsterdam, 1990.

\bibitem[She09]{Shelah:DependentContinued}
\bysame, \emph{Dependent first order theories, continued}, Israel Journal of
  Mathematics \textbf{173} (2009), 1--60,
  \href{http://dx.doi.org/10.1007/s11856-009-0082-1}{doi:10.1007/s11856-009-0082-1}.

\bibitem[Wag00]{Wagner:SimpleTheories}
Frank~O. \bgroup\scshape{}Wagner\egroup{}, \emph{Simple theories}, Kluwer
  Academic Publishers, 2000.

\end{thebibliography}
\end{document}